\documentclass[
11pt,                          
english                        
]{article}

\usepackage[english]{babel}    
\usepackage{amsmath}           
\usepackage{amssymb}           
\usepackage[utf8]{inputenc}    
\usepackage[T1]{fontenc}       
\usepackage[a4paper]{geometry} 
\usepackage{setspace}          
\usepackage{bbm}               
\usepackage{mathrsfs}          
\usepackage{mathtools}         
\usepackage{stmaryrd}          
\usepackage[amsmath,thmmarks,hyperref]{ntheorem} 
\usepackage{exscale}           
\usepackage[sort]{cite}        
\usepackage{xspace}            
\usepackage[activate={true,nocompatibility},
            final,tracking=true,kerning=true,
            spacing=true,factor=1100,
            stretch=10,shrink=10,expansion=false
           ]{microtype}
\microtypecontext{spacing=nonfrench}  
\usepackage[backref=page,      
           final=true,         
           hypertexnames=false,
           plainpages=false,              
           pdfpagelabels=true,            
           pdfencoding=auto,              
           unicode=true                   
                      ]{hyperref}         
\usepackage[shortcuts]{extdash}  

\author{
  \textbf{Matthias Schötz}\thanks{\texttt{matthias.schoetz@mathematik.uni-wuerzburg.de}} 
  \\[0.5cm]
  Institut für Mathematik \\
  Universität Würzburg
}

\makeatletter

\renewcommand{\mathbb}[1]{\mathbbm{#1}}

\numberwithin{equation}{section}

\renewcommand{\arraystretch}{1.2}

\let\originalleft\left
\let\originalright\right
\renewcommand{\left}{\mathopen{}\mathclose\bgroup\originalleft}
\renewcommand{\right}{\aftergroup\egroup\originalright}

\theoremheaderfont{\normalfont\bfseries}
\theorembodyfont{\itshape}
\newtheorem{lemma}{Lemma}[section]

\newtheorem{proposition}[lemma]{Proposition}
\newtheorem{theorem}[lemma]{Theorem}
\newtheorem{corollary}[lemma]{Corollary}
\newtheorem{definition}[lemma]{Definition}

\theorembodyfont{\rmfamily}

\newtheorem{example}[lemma]{Example}

\def\theorem@checkbold{}

\theoremheaderfont{\upshape}
\theorembodyfont{\normalfont}
\theoremstyle{nonumberplain}
\theoremseparator{:}
\theoremsymbol{\hbox{$\boxempty$}}
\newtheorem{proof}{\textbf{Proof}}

\newcommand{\I}              {\mathrm{i}}
\newcommand{\E}              {\mathrm{e}}
\newcommand{\cc}[1]          {\overline{{#1}}}
\newcommand{\Unit}           {\mathbb{1}}
\newcommand{\argument}       {\,\cdot\,}
\newcommand{\acts}            {\mathbin{\triangleright}}
\DeclarePairedDelimiter{\abs}{\lvert}{\rvert}
\DeclarePairedDelimiter{\norm}{\lVert}{\rVert}
\newcommand{\hilbert}[1]       {\mathfrak{#1}}
\DeclareMathOperator{\Var} {\mathrm{Var}}

\DeclarePairedDelimiter{\ordinaryIP}{\langle}{\rangle}
\DeclarePairedDelimiter{\ordinarySet}{\{}{\}}

\newcommand{\CC}{\mathbb{C}}
\newcommand{\RR}{\mathbb{R}}
\newcommand{\NN}{\mathbb{N}}
\newcommand{\neu}[2][]{\emph{#2}}
\newcommand{\dupr}[3][]{\ordinaryIP[#1]{\,#2 \,,\, #3\,}}
\newcommand{\skal}[3][]{\ordinaryIP[#1]{\,#2 \,#1|\, #3\,}}
\newcommand{\set}[3][]{\ordinarySet[#1]{\,#2 \;#1|\; #3\,}}
\newcommand{\Hermitian}{\textup{H}}
\newcommand{\States}{\mathcal{S}}
\newcommand{\PureStates}{\States_{\textup{p}}}
\newcommand{\Characters}{\mathcal{M}}
\newcommand{\cl}{\textup{cl}}
\newcommand{\seminorm}[3][]{\norm[#1]{#3}_{#2}}
\newcommand{\Adbar}{\mathcal{L}^*}
\newcommand{\Dom}[1][]{\mathcal{D}_{\smash{#1}}}
\newcommand{\Hilb}{\hilbert{H}}
\newcommand{\wk}{\textup{wk}}
\newcommand{\st}{\textup{st}}
\newcommand{\Gelfand}{\mathcal{I}}
\newcommand{\algebra}[1]{\mathcal{#1}}
\newcommand{\A}{\algebra{A}}
\newcommand{\Bounded}{\mathcal{B}}

\makeatother

\title{\texorpdfstring{On Characters and Pure States of $^*$-Algebras\\\small (Der Trilogie erster Teil)}{On Characters and Pure States of *-Algebras}}

\date{\today}

\begin{document}
\begin{onehalfspace}

\maketitle

\begin{abstract}
  It is easy to see that every character (i.e. unital $^*$\=/homomorphism to $\CC$)
  of a commutative unital associative $^*$\=/algebra is a pure state (i.e. extreme point in 
  the convex set of all normalized positive linear functionals). This article gives
  sufficient conditions for the converse to be true as well. In order to formulate
  these results together with similar ones, e.g. for locally convex $^*$\=/algebras, the notion
  of an abstract $O^*$\=/algebra (unital associative $^*$\=/algebra with an order defined by
  positive linear functionals) is introduced. Many concepts and intermediary results
  discussed here also apply to the non-commutative case.
\end{abstract}

\section{Introduction}
\label{sec:Introduction}
Let $\A$ be a unital associative $^*$\=/algebra over $\CC$, then denote by $\A^*$ the 
(algebraic) dual space of $\A$, consisting of all linear functionals to $\CC$, and
by $\dupr{\argument}{\argument}\colon \A^*\times \A \to \CC$ the dual pairing.
An \neu{algebraically positive linear functional} on $\A$ is an element 
$\omega\in\A^*$ for which $\dupr{\omega}{a^*a}\ge 0$ holds for all $a\in\A$.
An \neu{algebraic state} on $\A$ is an algebraically positive linear functionals
$\omega$ that is normalized to $\dupr{\omega}{\Unit} = 1$,
and an \neu{algebraic character} on $\A$ is a unital $^*$\=/homomorphism from $\A$ to $\CC$.
An \neu{extreme point} of a convex subset $C$ of a (real or complex) vector space is
an $e\in C$ with the property that $e=\lambda c_1 + (1-\lambda) c_2$ with
$\lambda \in\,\,]0,1[\,$ and $c_1,c_2 \in C$ implies $e=c_1=c_2$.
Clearly, the set of algebraic states on $\A$ is a convex subset of $\A^*$ and 
every algebraic character on $\A$ is an algebraic state, and even an extreme point 
in the set of all algebraic states (well-known, see Proposition~\ref{proposition:charactersAreExtreme}).

This raises the question
of whether or not every extreme point of the set of algebraic states is also an algebraic character
(at least in the commutative case).
Naturally, a similar question can also be asked if $\A$ is endowed with additional
structure, e.g. a locally convex topology: in this case one could consider extreme
points of the convex set of continuous algebraic states.
As it is not clear whether these are also extreme points in the larger set of all algebraic states,
this cannot simply be reduced to the previous problem.

Having a general characterization of characters as pure states is also interesting for several other reasons:
First, of course, this yields a possibility to prove existence of characters if one can
prove existence of pure states by applying the Krein-Milman theorem or generalizations thereof (like 
\cite{neighbourhoodsofextremepoints}, Theorem~3.6) to the set of states. Due to the Gel'fand
transformation, existence of characters is equivalent to existence of non-trivial representations
of a $^*$\=/algebra by complex-valued functions with the pointwise operations.
Moreover, the set of algebraic characters is quite obviously weak-$^*$-closed, which is not clear at all
for the set of pure algebraic states. If these sets coincides, this thus gives insight into the structure
of the convex set of states.

From the point of view of non-commutative geometry, pure states of a $^*$-algebra
can be seen as the non-commutative generalisation of characters, at least for those
classes of $^*$-algebras, for which the sets of pure states and characters
actually coincide under the additional assumption of commutativity.

In the case that $\A$ is a commutative unital Banach $^*$\=/algebra, it has been shown
directly  by R. S. Bucy and G. Maltese (in \cite{reptheoforposfun}, Theorem 2)
that every extreme point of the set of continuous algebraic states is also multiplicative,
hence a unital $^*$\=/homomorphism. The argument used there can be adapted to more general 
cases (see also \cite{gaur}), up to lmc $^*$\=/algebras, but always requires a
certain boundedness condition to be fulfilled (see Theorem \ref{theorem:bounded} here).

The main Theorem~\ref{theorem:unbounded} here gives a
sufficient condition for the sets of pure states and characters to coincide,
that can be applied to truly unbounded cases. For this to be applicable to as 
many different types of $^*$\=/algebras as possible, the notion of an abstract $O^*$\=/algebra
is introduced (Definition \ref{definition:abstractOstaralgebra}).

Another motivation for introducing (also non-commutative) abstract $O^*$\=/algebra comes from non-formal
deformation quantisation, where some recent examples of deformations of locally
convex $^*$-algebras outside the realm of $C^*$-algebras (or even lmc-$^*$-algebras)
have been constructed (like \cite{beiser.waldmann:2014a}, \cite{esposito.stapor.waldmann:2017a}
and \cite{schoetz.waldmann:2018a}). While these algebras can contain truely unbounded
elements, e.g. elements fulfilling canonical commutation relations, it seems
that there is not yet a general theory available that answers some questions
arising from their interpretation as observable algebras of physical systems,
e.g. concerning the sprectrum of an observable or the decomposition of arbitrary
states of the commutative, classical limits into pure states or characters.
Focusing not only on the topology of these algebras, but also on their order properties,
might help to solve these problems.

The article is organized as follows: Section~\ref{sec:Defs} gives the definition
of $O^*$\=/algebras and abstract $O^*$\=/algebras and discusses some basic constructions,
especially the (closed) GNS representation. In Section~\ref{sec:algebra},
some essentially algebraic properties of abstract $O^*$\=/algebras are examined and
used to prove that all characters of abstract $O^*$\=/algebras are pure states,
and that the converse is true in commutative symmetric abstract $O^*$\=/algebras 
(which, by definition, have many invertible elements). As the assumption of a
symmetric abstract $O^*$\=/algebra is rather strong, Section~\ref{sec:bounded}
seeks to replace this by a less restrictive one, mostly a condition for the
growth of powers of algebra elements that guarantees that all GNS representations
are by bounded operators. The argumentation up to that point is inspired by \cite{reptheoforposfun}
and \cite{gaur}. Finally, Section~\ref{sec:unbounded} uses techniques
from the theory of unbounded operator algebras (i.e. $O^*$\=/algebras)
to proof the main Theorem~\ref{theorem:unbounded},
which essentially states that all pure states of a commutative abstract $O^*$\=/algebra are characters,
if a growth condition (similar to the one in Nelson's theorem) on powers of 
algebra elements is fulfilled that guarantees 
that sufficiently many elements yield essentially self-adjoint operators
in every GNS representation.

\textbf{Notation}: 
A \neu{$^*$\=/algebra} $\A$ will always be understood to be defined over the field of complex numbers,
be associative and have a unit $\Unit$. Moreover, $\A_\Hermitian := \set{a\in \A}{a=a^*}$
is the real linear subspace of $\A$ of \neu{Hermitian} elements and its
convex cone of \neu{algebraically positive} elements 
(which might contain a nontrivial linear subspace) is
\begin{equation*}
  \A^{++}_\Hermitian := \set[\Big]{\sum\nolimits_{n=1}^N a_n^*a_n}{N\in\NN;\,a_1,\dots,a_N\in \A}\,,
\end{equation*}
and is indeed closed under multiplication with non-negative reals.
The real linear span of $\A^{++}_\Hermitian$ is $\A_\Hermitian$ because $4a = (a+\Unit)^2-(a-\Unit)^2$
for all $a\in \A_\Hermitian$, and so the (complex) linear span of $\A^{++}_\Hermitian$ is whole $\A$.
However, in most cases a different notion of positivity will be used, which
arises from a possibly larger convex cone of positive elements (that still spans whole $\A$).
Finally, sesquilinear maps, especially inner products, will always be 
antilinear in the first and linear in the second argument.

\textbf{Acknowledgements}:
I would like to thank Stefan Waldmann and Chiara Esposito (both University Würzburg)
for some helpful discussions and suggestions.

\section{\texorpdfstring{Abstract $O^*$\=/algebras}{Abstract O*-Algebras}}
\label{sec:Defs}
\begin{definition}
  A \neu{quasi-ordered $^*$\=/algebra} is a $^*$\=/algebra $\A$ endowed with a
  quasi-order $\lesssim$ (a reflexive and transitive relation) on $\A_\Hermitian$, 
  such that the following conditions are fulfilled for
  all $a,b\in\A_\Hermitian$ with $a\lesssim b$, all 
  $c\in\A_\Hermitian$ and all $d\in\A$:
  \begin{equation*}
    a+c \lesssim b+c\,, \quad\quad
    d^*a\, d \lesssim d^* b \,d \quad\quad \text{and}\quad\quad
    0\lesssim \Unit\,.
  \end{equation*}
  If $\lesssim$ is even a partial order (i.e. additionally antisymmetric), 
  then $\A$ is called an \neu{ordered $^*$\=/algebra}
  and we might write $\le$ instead of $\lesssim$.
  Moreover, if $\A$ is a quasi-ordered $^*$\=/algebra, then the convex cone of
  \neu{positive elements} in $\A$ is written as
  $\A^+_\Hermitian := \set{a\in\A_\Hermitian}{a\gtrsim 0}$.
\end{definition}
Note that $\A^{++}_\Hermitian \subseteq \A^+_\Hermitian$ holds for every quasi-ordered $^*$\=/algebra $\A$.
One important example of ordered $^*$\=/algebras are $O^*$\=/algebras (see e.g. \cite{schmuedgen}, Definitions~2.1.6 and 2.2.8
as well as the remark under Corollary~2.1.9, but note that the notation for seminorms is different):
\begin{definition} \label{definition:Ostar}
  Let $\Hilb$ be a Hilbert space and $\Dom\subseteq \Hilb$ a dense linear subspace.
  Then write $\Adbar(\Dom)$ for the $^*$\=/algebra of all adjointable endomorphisms 
  of $\Dom$, i.e. for the set of all (necessarily linear) $a\colon \Dom \to \Dom$ for which there exists
  a (necessarily unique and linear) $a^*\colon \Dom\to \Dom$ such
  that $\skal{\phi}{a\psi} = \skal{a^*\phi}{\psi}$ holds for all $\phi,\psi\in \Dom$.
  An $O^*$\=/algebra on $\Dom$ is defined as a unital $^*$\=/subalgebra $\A$ of $\Adbar(\Dom)$,
  and an order $\le$ on $\A_\Hermitian$ is defined by 
  \begin{equation*}
    a\le b \quad:\Leftrightarrow \quad \forall_{\phi\in \Dom}: \skal{\phi}{a\,\phi} \le \skal{\phi}{b\,\phi}.
  \end{equation*}
  Moreover, for every $a\in \A^+_\Hermitian$ a positive sesquilinear form $\skal{\argument}{\argument}_a$ on
  $\Dom$ is defined as $\skal{\phi}{\psi}_a := \skal{\phi}{a\,\psi}$ for all $\phi,\psi\in \Dom$
  and the corresponding seminorm is denoted by $\seminorm{a}{\argument}$. The $O^*$\=/algebra $\A$
  is called \neu{closed} if $\Dom$ is complete with respect to the locally convex topology
  $\tau_{\A}$ defined by the seminorms $\seminorm{a}{\argument}$ for all $a\in \A^+_\Hermitian$.
\end{definition}
One can check that, with this order, an $O^*$\=/algebra is indeed an ordered $^*$\=/algebra
(the order on the Hermitian elements is a partial one due to the Cauchy Schwarz inequality).
A short remark on the notion of the adjoint in an $O^*$\=/algebra $\A\subseteq \Adbar(\Dom)$
might be due as there obviously exists a relation between the adjoint endomorphism $a^* \in \A$
and the operator theoretic adjoint $a^\dagger$ of an element $a\in \A$: 

Every $a\in \Adbar(\Dom)$ is a closable operator on $\Hilb$ with domain $\Dom$.
The completion $\Dom[a^\cl]$ of $\Dom$ under the norm $\seminorm{\Unit+a^*a}{\argument}$ 
can be identified with the linear subspace of all $\smash{\hat{\phi}} \in \Hilb$ for which there
exists a sequence $(\phi_n)_{n\in \NN}$ in $\Dom$ that converges against $\smash{\hat{\phi}}$
with respect to the norm $\seminorm{}{\argument}$ on $\Hilb$ and is a Cauchy sequence with 
respect to $\seminorm{a^*a}{\argument}$. Clearly, $a$ is continuous as a map
from $\Dom$ with $\seminorm{\Unit+a^*a}{\argument}$ to $\Hilb$ with $\seminorm{}{\argument}$, and
its closure $a^\cl \colon \Dom[a^\cl]\to \Hilb$ is the continuous extension.
The adjoint operator $a^\dagger \colon \Dom[a^\dagger] \to \Hilb$, defined on
$\Dom[a^\dagger] := \set{\phi\in\Hilb}{\Dom \ni \psi \mapsto \skal{\phi}{a\,\psi} \in \CC \text{ is $\seminorm{}{\argument}$\=/continuous}\,}$
by the requirement that $\skal{a^\dagger \phi}{\psi} = \skal{\phi}{a\,\psi}$ for 
all $\phi\in\Dom[a^\dagger]$ and all $\psi \in \Dom$,
extends the adjoint endomorphism $a^*$ and also its closure $(a^*)^\cl \colon \Dom[(a^*)^\cl] \to \Hilb$.
The natural question to ask is whether or not $a^\dagger = (a^*)^\cl$. For Hermitian $a$ this is
equivalent to $a$ being essentially self-adjoint.

Moreover, out of an $O^*$\=/algebra $\A\subseteq \Adbar(\Dom)$ that is not closed one can 
construct a closed one by noting that all $a\in \A$ are continuous with respect to $\tau_{\A}$
and thus extend to continuous adjointable endomorphisms of the completion of $\Dom$ under $\tau_{\A}$,
which can be identified with $\Dom^\cl := \bigcap_{a\in \A} \Dom[a^\cl] \subseteq \Hilb$ 
carrying the projective limit topology (see \cite{schmuedgen}, Lemma 2.2.9 and the discussion 
thereafter; the different systems of seminorms used here and in \cite{schmuedgen} are 
equivalent).

The order on an $O^*$\=/algebra is especially well-behaved in so far as it comes from a set
of algebraically positive linear functionals. Generalizing this leads to the notion of
an abstract $O^*$\=/algebra. Let $\A$ be an arbitrary quasi-ordered $^*$\=/algebra, then its
dual space $\A^*$ is naturally endowed with:
\begin{itemize}
  \item An \neu{antilinear involution} $\argument^* \colon \A^*\to\A^*$ given by
    $\omega^*(a) := \cc{\omega(a^*)}$ for all $\omega\in \A^*$ and all $a\in \A$.
  \item An \neu{$\A$-bimodule structure} $\A \times \A^* \times \A \to \A^*$ given by
    $\dupr{b\cdot\omega\cdot c}{a} := \dupr{\omega}{cab}$ for all $\omega \in \A^*$ and all $a,b,c\in\A$.
  \item An \neu{$\A$-monoid action} $\acts:\A \times \A^*\to \A^*$ given by
    $a\acts \omega := a\cdot \omega \cdot a^*$ for all $\omega \in \A^*$ and all $a\in \A$.
  \item A \neu{partial order} $\le$ on $\A^*_\Hermitian := \set{\omega\in\A^*}{\omega^* = \omega}$
    given by 
    \begin{equation*}
      \omega \le \rho \quad:\Leftrightarrow\quad \forall_{a\in \A^+_\Hermitian}:\dupr{\omega}{a}\le\dupr{\rho}{a}
    \end{equation*}
    for all $\omega,\rho \in \A^*_\Hermitian$ and fulfilling
    \begin{equation*}
      \omega+\kappa \le \rho + \kappa 
      \quad\quad \text{and}\quad\quad
      a\acts \omega \le a \acts \rho 
    \end{equation*}
    for all $\omega,\rho,\kappa \in \A^*_\Hermitian$ with $\omega\le \rho$ 
    and all $a\in \A$. Thus define $\A^{*,+}_\Hermitian := \set{\omega\in\A^*_\Hermitian}{\omega \ge 0}$.
\end{itemize}
Clearly, every $\omega \in \A^{*,+}_\Hermitian$ is an algebraically positive linear functional
because $\A^{++}_\Hermitian \subseteq \A^+_\Hermitian$.
Note that due to the existence of a unit $\Unit\in \A$ and due to the Cauchy Schwarz inequality 
\begin{equation*}
  \abs{\dupr{\omega}{a^*b}}^2 \le \dupr{\omega}{a^*a}\dupr{\omega}{b^*b}
\end{equation*}
for all $a,b\in \A$ and all algebraically positive $\omega\in \A^*_\Hermitian$,
the relation $\le$ on $\A^*_\Hermitian$ is indeed a partial order and not just a quasi-order.
Moreover, every algebraically positive $\omega\in \A^*$ actually fulfils $\omega^*=\omega$, hence $\omega\in \A^*_\Hermitian$,
because $4\dupr{\omega}{a} = \dupr{\omega}{(a+\Unit)^2} - \dupr{\omega}{(a-\Unit)^2} \in \RR$
for all $a\in\A_\Hermitian$.
\begin{definition}\label{definition:abstractOstaralgebra}
  An \neu{abstract $O^*$\=/algebra} is a tuple $(\A,\Omega)$ of a
  quasi-ordered  $^*$\=/algebra $\A$ and a linear subspace and $\A$-subbimodule
  $\Omega\subseteq\A^*$ that is stable under the antilinear involution 
  $\argument^*$, i.e. $\omega^* \in \Omega$ if $\omega\in \Omega$,
  and is compatible with the order on $\A$ in the following way:
  Define the real linear subspace $\Omega_\Hermitian := \Omega \cap \A^*_\Hermitian$
  of \neu{Hermitian linear functionals} of $(\A,\Omega)$ and the convex cone
  $\Omega^+_\Hermitian := \Omega \cap \A^{*,+}_\Hermitian$ of \neu{positive linear functionals}
  of $(\A,\Omega)$, where $\A^*_\Hermitian$ and $\A^{*,+}_\Hermitian$ are like above.
  Then it is required that $\Omega$ is the linear span of $\Omega^+_\Hermitian$
  and that 
  \begin{equation*}
    \A^+_\Hermitian
    =
    \set[\big]{
      a\in\A_\Hermitian
    }{
      \forall_{\omega\in\Omega^+_\Hermitian}:
      \dupr{\omega}{a} \ge 0
    }\,.
  \end{equation*}
\end{definition}
So every positive linear functional $\omega$ of an abstract $O^*$\=/algebra $(\A,\Omega)$
is also algebraically positive, but conversely, an algebraically positive linear functional
$\omega$ on $\A$ need not be in $\Omega^+_\Hermitian$, as neither $\omega \in \A^+_\Hermitian$
nor $\omega \in \Omega$ are guaranteed.
By the above definition, the order on $\A_\Hermitian$ determines the order on $\Omega_\Hermitian$
and vice versa. Because of this, the construction of an abstract $O^*$\=/algebra is rather
easy provided one has a distinguished set of algebraically positive linear functionals:
\begin{proposition} \label{proposition:construction}
  Let $\A$ be a $^*$\=/algebra and $P^+_\Hermitian \subseteq \A^*_\Hermitian$ a
  set of algebraically positive linear functionals that is stable under the monoid
  action $\acts$ of $\A$, i.e. $a\acts \omega \in P^+_\Hermitian$ for all $a\in\A$ 
  and all $\omega \in P^+_\Hermitian$. Define a relation $\lesssim$ on $\A_\Hermitian$ by
  \begin{equation*}
    a\lesssim b \quad:\Leftrightarrow \quad\forall_{\omega\in P^+_\Hermitian}:\dupr{\omega}{a} \le \dupr{\omega}{b}\,,
  \end{equation*}
  then $\lesssim$ is a quasi-order that turns $\A$ into a quasi-ordered $^*$\=/algebra.
  Moreover, let $\Omega$ be the linear subspace of $\A^*$ generated by $P^+_\Hermitian$,
  then $(\A,\Omega)$ is an abstract $O^*$\=/algebra and 
  $\Omega^+_\Hermitian$ is the weak-$^*$-closure in $\Omega$ of 
  $\set[\big]{\sum_{n=1}^N \omega_n}{N\in\NN_0;\, \omega_1,\dots,\omega_N\in P^+_\Hermitian}$,
  the convex cone generated by $P^+_\Hermitian$.
\end{proposition}
\begin{proof}
  It is clear that $\A$ with $\lesssim$ is a quasi-ordered $^*$\=/algebra, and $(\A,\Omega)$
  is then an abstract $O^*$\=/algebra:
  \begin{equation*}
    a \cdot \omega \cdot b^* = \frac{1}{4} \sum\nolimits_{k=0}^3 \I^k \big((a+\I^kb)\acts \omega\big) \in \Omega
  \end{equation*}
  holds for all $a,b\in \A$ and all $\omega\in \Omega$ and proves that $\Omega$ is an
  $\A$-subbimodule. $\Omega$ is stable under $\argument^*$
  because $P^+_\Hermitian$ is, and it determines the order on $\A$ by construction.
  As $P^+_\Hermitian \subseteq \Omega^+_\Hermitian$ by construction, the linear span
  of $\Omega^+_\Hermitian$ is $\Omega$. Finally, as $\A^{*,+}_\Hermitian$ is a weak-$^*$-closed
  convex cone, $\Omega^+_\Hermitian$ is a convex cone that is weak-$^*$-closed in $\Omega$
  and thus the weak-$^*$-closure in $\Omega$ of 
  $P^{+,\textup{cn}}_\Hermitian := \set[\big]{\sum_{n=1}^N \omega_n}{N\in\NN_0;\, \omega_1,\dots,\omega_N\in P^+_\Hermitian}$
  is a subset of $\Omega^+_\Hermitian$. Conversely, if $\rho \in \Omega_\Hermitian$
  is not in the weak-$^*$-closure of
  $P^{+,\textup{cn}}_\Hermitian$,
  then by the separation theorem for the closure of the convex set $P^{+,\textup{cn}}_\Hermitian$
  from the compact point $\rho$ there exists an $a\in \A_\Hermitian$ such that 
  $\dupr{\omega}{a} \ge \dupr{\rho}{a}+1$
  for all $\omega\in P^{+,\textup{cn}}_\Hermitian$ because $P^{+,\textup{cn}}_\Hermitian$ is convex
  (see e.g. \cite{koethe}, §20.7 (2), and use that the weak-$^*$-topology on $\Omega_\Hermitian$ is
  the one induced by the dual pairing with $\A_\Hermitian$),
  and even $\dupr{\omega}{a} \ge 0 \ge \dupr{\rho}{a}+1$ because $P^{+,\textup{cn}}_\Hermitian$ is a cone,
  so $a\in \A^+_\Hermitian$ but $\dupr{\rho}{a}<0$, i.e. $\rho \notin \Omega^+_\Hermitian$.
\end{proof}
Of course, the order on an $O^*$\=/algebra was defined just like this starting with the set of positive 
linear functionals coming from inner products.
By definition, an abstract $O^*$\=/algebra with a non-trivial order has non-trivial positive linear
functionals, hence non-trivial representations as $O^*$\=/algebras, e.g. the well-known GNS
representations (see \cite{schmuedgen}, Theorems 8.6.2 and 8.6.4):
\begin{definition} \label{definition:gns}
  Let $\A$ be a $^*$\=/algebra and $\omega$ an algebraically positive linear functional on $\A$, then
  the \neu{Gel'fand ideal} associated to $\omega$ 
  is defined as $\Gelfand_\omega := \set{a\in\A}{\dupr{\omega}{a^*a} = 0}$,
  and the Hilbert space $\Hilb_\omega$ as the completion of $\A / \Gelfand_\omega$ with
  inner product $\skal{[a]_\omega}{[b]_\omega}_\omega := \dupr{\omega}{a^*b}$ for all 
  $[a]_\omega,[b]_\omega\in \A / \Gelfand_\omega$
  having representatives $a,b\in \A$. We can naturally identify $\A/\Gelfand_\omega$ with a
  dense linear subspace $\Dom[\omega]$ of $\Hilb_\omega$ and construct a linear map
  $[\argument]_\omega\colon \A \to \Dom[\omega]$ that assigns to every $a\in \A$
  its equivalence class $[a]_\omega \in \Dom[\omega]$ under this identification,
  as well as a unital $^*$\=/homomorphism 
  $\pi_\omega \colon \A \to \Adbar(\Dom[\omega])$ by $\pi_\omega(a)[b]_\omega := [ab]_\omega$,
  that maps the $^*$\=/algebra $\A$ onto an $O^*$\=/algebra $\pi_\omega(\A) \subseteq \Adbar(\Dom[\omega])$
  and describes the \neu{GNS representation of $\A$ associated to $\omega$}.
  
  Moreover, let $\Dom[\omega]^\cl := \bigcap_{a\in \A} \Dom[\omega,\pi_\omega(a)^\cl] \subseteq \Hilb_\omega$,
  which can be identified with the completion of $\Dom[\omega]$ under $\tau_{\pi_\omega(\A)}$,
  and define $\pi_\omega^\cl(a) \colon \Dom[\omega]^\cl \to \Dom[\omega]^\cl$
  as the continuous extension of $\pi_\omega(a)$ for every $a\in \A$. Then
  $\pi_\omega^\cl\colon \A \to \Adbar(\Dom[\omega]^\cl)$ describes the 
  \neu{closed GNS representation of $\A$ associated to $\omega$} and is
  a unital $^*$\=/homomorphism of $\A$ onto a closed $O^*$\=/algebra $\pi_\omega^\cl(\A) \subseteq \Adbar(\Dom[\omega]^\cl)$.
\end{definition}
\begin{lemma} \label{lemma:gnspositive}
  Let $\A$ be a quasi-ordered $^*$\=/algebra and $\omega\in\A^{*,+}_\Hermitian$, then the two
  GNS representations $\pi_\omega\colon \A \to \Adbar(\Dom[\omega])$ and 
  $\pi_\omega^\cl\colon \A \to \Adbar(\Dom[\omega]^\cl)$ are positive, i.e.
  $\pi_\omega(a) \ge 0$ and $\pi_\omega^\cl(a) \ge 0$ for all $a\in \A^+_\Hermitian$.
\end{lemma}
\begin{proof}
  As $\skal{[b]_\omega}{\pi_\omega(a)\,[b]_\omega}_\omega = \dupr{\omega}{b^*a\,b} \ge 0$
  for all $[b]_\omega\in\Dom[\omega]$ and all $a\in \A^+_\Hermitian$,
  this is true for the ordinary GNS representation $\pi_\omega$.
  Using that
  $\Dom[\omega]$ is $\tau_{\pi_\omega^\cl(\A)}$-dense in $\Dom[\omega]^\cl$ and
  $\pi_\omega^\cl(a)$ is $\tau_{\pi_\omega^\cl(\A)}$\=/continuous for all $a\in \A^+_\Hermitian$,
  this inequality extends to the closed GNS representation 
  $\pi^\cl_\omega$.
\end{proof}
The definitions of states, pure states and characters of an abstract $O^*$\=/algebra
are:
\begin{definition}
  Let $(\A,\Omega)$ be an abstract $O^*$\=/algebra. Then 
  \begin{align*}
    \States(\A,\Omega) &:= \set[\big]{\omega \in \Omega^+_\Hermitian}{\dupr{\omega}{\Unit}= 1} \\
    \quad\quad\text{and}\quad\quad
    \PureStates(\A,\Omega) &:= \set[\big]{\omega \in \States(\A,\Omega)}{\omega\text{ is an extreme point of }\States(\A,\Omega)}
  \end{align*}
  are the sets of \neu{states} and \neu{pure states} of $(\A,\Omega)$, respectively, and
  \begin{equation*}
    \Characters(\A,\Omega) 
    := 
    \set[\big]{
      \omega \in \States(\A,\Omega)
    }{
      \omega\text{ is multiplicative, i.e. }\dupr{\omega}{ab}=\dupr{\omega}{a}\dupr{\omega}{b}\text{ for all }a,b\in\A
    }
  \end{equation*}
  is the set of \neu{characters} of $(\A,\Omega)$.
\end{definition}
So a character of an abstract $O^*$-algebra $(\A,\Omega)$ is also required to be
positive with respect to the ordering on the quasi-ordered $^*$-algebra $\A$,
and not just algebraically positive (which would be an immediate consequence of its
multiplicativity).

If $\A$ is an arbitrary $^*$\=/algebra, one can choose $P^+_\Hermitian$ 
as the set of all algebraically positive linear functionals on $\A$ and construct an
abstract $O^*$\=/algebra $(\A,\Omega)$ like in Proposition~\ref{proposition:construction}.
Similarly, if $\A$ is a $^*$\=/algebra carrying a topology that makes 
the operators of left- and right-multiplication on $\A$ with fixed elements of $\A$
continuous, then one can choose $P^+_\Hermitian$ as the set of all continuous algebraically 
positive linear functionals on $\A$ and again construct an abstract $O^*$\=/algebra $(\A,\Omega)$ as before.
Finding suitable sufficient conditions for a commutative abstract $O^*$\=/algebra $(\A,\Omega)$, such that 
\begin{equation*}
  \PureStates(\A,\Omega) = \Characters(\A,\Omega)
\end{equation*}
holds, then answers the questions raised in the introduction with respect to pure states and characters.
\section{Algebraic properties} \label{sec:algebra}
The following definition will be extremely useful:
\begin{definition}
  Let $\A$ be a $^*$\=/algebra, $\omega$ an algebraic state on $\A$
  and $a\in \A$. Then define the \neu{variance of $\omega$ on $a$}:
  \begin{equation*}
    \Var_\omega(a) 
    := 
    \dupr[\big]{\omega}{\big(a-\dupr{\omega}{a}\Unit\big)^*\big(a-\dupr{\omega}{a}\Unit\big)}
    =
    \dupr{\omega}{a^*a} - \abs{\dupr{\omega}{a}}^2\,.
  \end{equation*}
\end{definition}
Note that $\Var_\omega(a) \ge 0$ and
\begin{equation*}
  \abs[\big]{\dupr{\omega}{b^*a}-\cc{\dupr{\omega}{b}}\dupr{\omega}{a}}^2
  =
  \abs[\big]{\dupr[\big]{\omega}{\big(b-\dupr{\omega}{b}\Unit\big)^*\big(a-\dupr{\omega}{a}\Unit\big)}}^2
  \le
  \Var_\omega(b) \Var_\omega(a)
\end{equation*}
holds for all algebraic states $\omega$ on $\A$ and all elements $a,b\in \A$
due to the Cauchy Schwarz inequality. This proves:
\begin{lemma} \label{lemma:var}
  Let $\A$ be a $^*$\=/algebra, $\omega$ an algebraic state on $\A$
  and $a\in \A$ with $\Var_\omega(a) = 0$, then
  \begin{equation*}
    \dupr{\omega}{ba} = \dupr{\omega}{b} \dupr{\omega}{a}
  \end{equation*}
  holds for all $b\in \A$. Thus $\omega$ is multiplicative if and only if
  $\Var_\omega(a) = 0$ for all $a\in \A$.
\end{lemma}
The first consequence is the following essentially well-known result:
\begin{proposition} \label{proposition:charactersAreExtreme}
  On an abstract $O^*$\=/algebra $(\A,\Omega)$, every character is a pure state, i.e.
  \begin{equation*}
    \Characters(\A,\Omega) \subseteq \PureStates(\A,\Omega)\,.
  \end{equation*}
\end{proposition}
\begin{proof}
  For all $\rho_1,\rho_2\in \States(\A,\Omega)$, all $\lambda\in [0,1]$ and all $a\in \A$,
  one can check that the identity
  \begin{equation*}
    \Var_{\lambda\rho_1 + (1-\lambda)\rho_2}(a)
    =
    \lambda \Var_{\rho_1}(a) + (1-\lambda) \Var_{\rho_2}(a) + \lambda(1-\lambda)\abs{\dupr{\rho_1-\rho_2}{a}}^2
  \end{equation*}
  holds. So if $\omega\in\States(\A,\Omega)$ is even a character of $(\A,\Omega)$
  and $\rho_1,\rho_2\in\States(\A,\Omega)$ fulfil
  $\omega = \lambda \rho_1 + (1-\lambda)\rho_2$ with  $\lambda\in\,\,]0,1[\,$, then $\Var_\omega(a)=0$ and 
  $\Var_{\rho_1}(a), \Var_{\rho_2}(a) \ge 0$ for all $a\in \A$
  imply that $\abs{\dupr{\rho_1-\rho_2}{a}}^2 = 0$ for all $a\in \A$, hence $\omega = \rho_1 = \rho_2$.
  We conclude that $\omega$ is an extreme point of $\States(\A,\Omega)$.
\end{proof}
The previous observation from Lemma~\ref{lemma:var}, that vanishing variance of a state $\omega$ implies that $\omega$ is
multiplicative, can even be strengthened:
\begin{lemma} \label{lemma:varquadrat}
  If $\A$ is a $^*$\=/algebra and $\omega$ an algebraic state on $\A$, then
  $\omega$ is multiplicative if and only if $\Var_\omega(a^2) = 0$
  holds for all $a\in\A_\Hermitian$.
\end{lemma}
\begin{proof}
  This condition is clearly necessary, but also sufficient: 
  If $\Var_\omega(a^2) = 0$ for all $a\in\A_\Hermitian$, then also
  $\Var_\omega((a\pm\Unit)^2) = 0$ for all $a\in\A_\Hermitian$, thus
  \begin{equation*}
    4\dupr{\omega}{a^2} = 
    \dupr{\omega}{a(a+\Unit)^2} - \dupr{\omega}{a(a-\Unit)^2}
    = \dupr{\omega}{a} \dupr{\omega}{(a+\Unit)^2} - \dupr{\omega}{a}\dupr{\omega}{(a-\Unit)^2}
    = 4\dupr{\omega}{a}^2
  \end{equation*}
  due to Lemma~\ref{lemma:var}, which proves $\Var_\omega(a) = 0$ for all $a\in\A_\Hermitian$.
  As every element of $\A$ can be expressed as a linear combination of Hermitian elements,
  it follows from Lemma~\ref{lemma:var} that $\omega$ is multiplicative.
\end{proof}
The essential property of pure states that we will have to exploit is the following:
\begin{lemma} \label{lemma:extremal}
  If $(\A,\Omega)$ is an abstract $O^*$\=/algebra, $\omega\in \PureStates(\A,\Omega)$
  and $\rho \in \Omega^+_\Hermitian$ such that
  $\rho \le \omega$, then $\rho = \dupr{\rho}{\Unit}\,\omega$.
\end{lemma}
\begin{proof}
  If $\dupr{\rho}{\Unit} = 0$, then $\rho = 0$ due to the Cauchy Schwarz inequality, and
  $\rho = \dupr{\rho}{\Unit}\,\omega$ is trivial. 
  If $\dupr{\rho}{\Unit} = 1$, then
  $\dupr{\omega-\rho}{\Unit} = 0$ together with $\omega-\rho\in \Omega^+_\Hermitian$ 
  show that $\omega = \rho$ and again $\rho = \dupr{\rho}{\Unit}\,\omega$ is trivial.
  Otherwise, let $\lambda := \dupr{\rho}{\Unit} \in\,\,]0,1[\,$, then
  $\omega = \lambda \big(\rho / \lambda\big) + (1-\lambda) \big((\omega-\rho)/(1-\lambda)\big)$ with
  states $\rho / \lambda$ and $(\omega-\rho)/(1-\lambda)$ implies $\omega = \rho / \lambda$.
\end{proof}
\begin{proposition} \label{proposition:extremal}
  Let $(\A,\Omega)$ be an abstract $O^*$\=/algebra, $\omega\in \PureStates(\A,\Omega)$
  and $\algebra{B}\subseteq \A$ a unital $^*$\=/subalgebra such that for every $a\in \algebra{B}_\Hermitian$
  there exists a $C_a\in[0,\infty[\,$ for which $a\acts \omega \le C_a\,\omega$ holds.
  Then $\omega$ is multiplicative on $\algebra{B}$.
\end{proposition}
\begin{proof}
  Given $a\in \algebra{B}_\Hermitian$ and a corresponding $C_a\in[0,\infty[\,$, then we can
  assume without loss of generality that $C_a>0$, in which case it follows from
  the previous Lemma~\ref{lemma:extremal} that 
  $C_a^{-1}\,(a\acts\omega) = C_a^{-1}\,\dupr{a\acts\omega}{\Unit}\,\omega$.
  Evaluating this on $a^2$ yields $\dupr{\omega}{a^4} = \dupr{\omega}{a^2}^2$ and
  thus $\Var_\omega(a^2)=0$. By Lemma~\ref{lemma:varquadrat}, $\omega$ is multiplicative on
  $\algebra{B}$.
\end{proof}
It might be worth mentioning that for every commutative $^*$-algebra $\A$ and every algebraic state $\omega$
on $\A$ there exists the largest unital $^*$\=/subalgebra of $\A$ on which $\omega$ is multiplicative,
namely the set $\set{a\in\A}{\Var_\omega(a)=0}$, which can be confirmed to be a unital subalgebra of $\A$
using Lemma~\ref{lemma:var} and is even a unital $^*$-subalgebra because $\Var_\omega(a) = \Var_\omega(a^*)$
for all $a\in\A$ due to the commutativity of $\A$.
As a first application we can now show that in a symmetric commutative abstract $O^*$\=/algebra,
the set of characters coincides with the set of pure states:
\begin{definition}
  An abstract $O^*$\=/algebra $(\A,\Omega)$ is called \neu{symmetric}, if for every $a\in \A_\Hermitian$
  there exists a multiplicative inverse of $\Unit+a^2$ in $\A$.
\end{definition}
Note that an abstract $O^*$\=/algebra $(\A,\Omega)$ is symmetric if and only if for every
$a\in \A_\Hermitian$ and every $\lambda \in \CC\backslash \RR$ there exists a multiplicative
inverse of $a-\Unit \lambda$. So the above definition is completely analogous
to the one of a symmetric $^*$\=/algebra in \cite{schmuedgen}, Chapter~1.4.
This condition has also occured in the literature before, e.g. in a similar way
in \cite{kaplansky}, Theorem~26 as a condition that assures that a Baer ring is 
a Baer $^*$-ring.
\begin{theorem} \label{theorem:symmetric}
  Let $(\A,\Omega)$ be a symmetric commutative abstract $O^*$\=/algebra, then 
  $\PureStates(\A,\Omega) = \Characters(\A,\Omega)$.
\end{theorem}
\begin{proof}
  Proposition~\ref{proposition:charactersAreExtreme} already shows that
  $\PureStates(\A,\Omega) \supseteq \Characters(\A,\Omega)$ and it remains
  to show that every pure state $\omega$ of $(\A,\Omega)$ is multiplicative.
  So let $a\in \A_\Hermitian$ be given and write $b := (\Unit+a^2)^{-1}$,
  then $b^* = b^*(\Unit+a^2)\,b = \big(b\,(\Unit+a^2)\big)^*b = b$.
  Assume $\dupr{\omega}{b^2} = 0$, then $\abs{\dupr{\omega}{b}}^2 \le \dupr{\omega}{b^2} = 0$
  implies $\dupr{\omega}{b} = \Var_\omega(b) = 0$
  and $1=\dupr{\omega}{(\Unit+a^2)\,b} = \dupr{\omega}{\Unit+a^2}\dupr{\omega}{b} = 0$
  by Lemma~\ref{lemma:var} yields a contradiction, so $\dupr{\omega}{b^2} > 0$. 
  
  Now observe that $b\acts \omega \le \omega$ because 
  $ \dupr{b\acts \omega}{c}
    \le
    \dupr{b\acts \omega}{c+2a^2c+a^4c}
    =
    \dupr{\omega}{c}$
  for all $c\in \A^+_\Hermitian$,
  hence $b\acts \omega = \dupr{\omega}{b^2}\,\omega$ by Lemma~\ref{lemma:extremal}.
  It follows that $a\acts \omega = \dupr{\omega}{b^2}^{-1}\,(ab\acts\omega) \le \dupr{\omega}{b^2}^{-1}\,\omega$,
  because
  $ \dupr{ab\acts \omega}{c} 
    \le
    \dupr{b\acts \omega}{c+2a^2c+a^4c}
    =
    \dupr{\omega}{c}
  $
  holds for all $c\in \A^+_\Hermitian$. By Proposition~\ref{proposition:extremal},
  $\omega$ is multiplicative on $\A$.
\end{proof}
However, the assumption of a symmetric commutative abstract $O^*$\=/algebra
is a rather strong one. In the following, similar theorems for more general classes of
algebras will be proven.
\section{Bounded states} \label{sec:bounded}
Another possibility to show that all pure states of certain commutative abstract $O^*$\=/algebras
are characters, is by exploiting some boundedness condition:
\begin{definition}
  Let $(\A,\Omega)$ be an abstract $O^*$\=/algebra and $\omega \in \States(\A,\Omega)$,
  then
  $\seminorm{\omega,\infty}{\argument}\colon \A \to [0,\infty]$ is defined as
  \begin{equation*}
    a\mapsto \seminorm{\omega,\infty}{a} := \sup_{b\in\A, \dupr{\omega}{b^*b}=1} \sqrt{\dupr{b\acts\omega}{a^*a}} \in [0,\infty]\,.
  \end{equation*}
  Moreover, given $a\in \A$, then $\omega$ is said to be a \neu{bounded state for $a$} if $\seminorm{\omega,\infty}{a} < \infty$.
  The set of all $a\in \A$ for which $\omega$ is bounded will be denoted by 
  $\Bounded_\omega(\A,\Omega)$ and $\Bounded(\A,\Omega) := \bigcap_{\omega\in \States(\A,\Omega)} \Bounded_\omega(\A,\Omega)$ is the set of all
  \neu{bounded} elements of $\A$.
\end{definition}
Note that $\Bounded(\A,\Omega)$ and $\Bounded_\omega(\A,\Omega)$ for every 
$\omega \in \States(\A,\Omega)$ are unital $^*$\=/subalgebras and that $\seminorm{\omega,\infty}{\argument}$ is a 
$C^*$\=/seminorm on $\Bounded_\omega(\A,\Omega)$. 
In fact, $\seminorm{\omega,\infty}{\argument}$ is nothing but the operator norm in the GNS representation
of $\A$ associated to $\omega$. This also shows that $\seminorm{\omega,\infty}{a}<\infty$ if and only if 
$a$ is represented by a bounded operator and that
$\dupr{b\acts \omega}{a^*a} \le \seminorm{\omega,\infty}{a}^2 \dupr{\omega}{b^*b}$
holds for all $a\in \Bounded_\omega(\A,\Omega)$ and all $b\in \A$.

In the following, we will often have to distinguish two cases: Either
$\dupr{\omega}{a} = 0$ for a state $\omega$ and a positive algebra element $a$, typically of the form
$a=b^*b$, then everything is trivial; or $\dupr{\omega}{a^n}>0$ for all $n\in \NN$:
\begin{lemma} \label{lemma:growth}
  Let $\A$ be a quasi-ordered $^*$\=/algebra, $\omega \in \A^{*,+}_\Hermitian$  and
  $a\in\A^+_\Hermitian$. Then $\dupr{\omega}{a^n} = 0$ for one $n\in \NN$ implies $\dupr{\omega}{a^n} = 0$
  for all $n\in \NN$ and also $\Var_\omega(a) = 0$.
  Otherwise, $\dupr{\omega}{a^n} > 0$ and
  \begin{align*}
    \frac{\dupr{\omega}{a^n}}{\dupr{\omega}{a^{n-1}}} &\le \frac{\dupr{\omega}{a^{n+1}}}{\dupr{\omega}{a^{n}}} \\
    \text{as well as}\quad\quad
    \dupr[\big]{\omega}{a^n}^{\frac{1}{n}} &\le \dupr[\big]{\omega}{a^{n+1}}^\frac{1}{n+1}
  \end{align*}
  hold for all $n\in \NN$.
\end{lemma}
\begin{proof}
  The sesquilinear form $\A^2\ni(b,c) \mapsto \dupr{\omega}{b^*a\,c} \in \CC$ is positive because $a\in \A^+_\Hermitian$,
  which yields $\abs{\dupr{\omega}{b^*a\,c}}^2 \le \dupr{\omega}{b^*a\,b}\dupr{\omega}{c^*a\,c}$
  for all $b,c\in \A$ by the Cauchy Schwarz inequality. So $\dupr{\omega}{a^{m-1} a^{m}}^2 \le \dupr{\omega}{a^{2m-2}} \dupr{\omega}{a^{2m}}$
  and $\dupr{\omega}{a^{m-1} a\, a^m}^2 \le \dupr{\omega}{a^{2m-1}} \dupr{\omega}{a^{2m+1}}$ hold
  for all $m\in \NN$ and show that $\dupr{\omega}{a^n}^2 \le \dupr{\omega}{a^{n-1}}\dupr{\omega}{a^{n+1}}$
  for all odd and all even $n\in \NN$, hence for all $n\in \NN$.
  Especially if $\dupr{\omega}{a^{n-1}}=0$ or $\dupr{\omega}{a^{n+1}}=0$ then also
  $\dupr{\omega}{a^n} = 0$. By induction it follows that $\dupr{\omega}{a^n} = 0$ for one $n\in \NN$ implies $\dupr{\omega}{a^n} = 0$
  for all $n\in \NN$, and then also $\Var_\omega(a) = 0$.

  Otherwise $\dupr{\omega}{a^n} > 0$ for all $n\in \NN$, because $a^{2m} = (a^m)^*(a^m)$
  and $a^{2m+1} = (a^m)^* a \,(a^m)$ are positive for all $m\in \NN_0$. 
  The estimate for quotients has already been proven, the one for roots is surely true if $n=1$,
  in which case it is just the Cauchy Schwarz inequality again. Now assume that it
  holds for one $n\in \NN$, then
  \begin{equation*}
    \dupr[\big]{\omega}{a^{n+1}}^{\frac{1}{n+1}} 
    \le
    \frac{\dupr{\omega}{a^{n+1}}}{ \dupr{\omega}{a^{n}} }
    \le
    \frac{\dupr{\omega}{a^{n+2}}}{ \dupr{\omega}{a^{n+1}} }\,,
  \end{equation*}
  which then implies
  $\dupr{\omega}{a^{n+1}}^{1/(n+1)} \le \dupr{\omega}{a^{n+2}}^{1/(n+2)}$.
\end{proof}
In the case of commutative abstract $O^*$\=/algebras, $\seminorm{\omega,\infty}{a}$ only depends on $\omega$ and
powers of $a$ and $a^*$:
\begin{proposition}
  Let $(\A,\Omega)$ be a commutative abstract $O^*$\=/algebra and $\omega \in \States(\A,\Omega)$,
  then
  \begin{equation*}
    \seminorm{\omega,\infty}{a} 
    = 
    \sup_{n\in \NN} \,\dupr[\big]{\omega}{(a^*a)^n}^{\frac{1}{2n}} 
    =
    \lim_{n\to\infty} \dupr[\big]{\omega}{(a^*a)^n}^{\frac{1}{2n}}
    \in [0,\infty]
  \end{equation*}
  holds for all $a\in \A$.
\end{proposition}
\begin{proof}
  The second identity is clear because $n\mapsto \dupr{\omega}{(a^*a)^n}^{1/(2n)}$ is non-decreasing
  by the previous Lemma~\ref{lemma:growth}. Define the shorthand 
  $\seminorm{\omega,\infty}{a}' := \sup_{n\in \NN} \dupr{\omega}{(a^*a)^{n}}^{1/(2n)}$,
  then
  \begin{equation*}
  \dupr{b\acts\omega}{a^*a}
  \le
  \dupr{b\acts\omega}{(a^*a)^m}^{\frac{1}{m}}
  \le
  \dupr{\omega}{(b^*b)^2}^{\frac{1}{2m}}\dupr{\omega}{(a^*a)^{2m}}^{\frac{1}{2m}}
  \le
  \dupr{\omega}{(b^*b)^2}^{\frac{1}{2m}}\,\big(\seminorm{\omega,\infty}{a}'\big)^2
  \end{equation*}
  holds for all $a, b\in \A$ with $\dupr{\omega}{b^*b}=1$ and all $m\in \NN$
  by Lemma~\ref{lemma:growth} again and the Cauchy Schwarz inequality, 
  thus $\seminorm{\omega,\infty}{a} \le \seminorm{\omega,\infty}{a}'$.
  Conversely, if $\seminorm{\omega,\infty}{a} < \infty$, then
  $\dupr{\omega}{(a^*a)^n} \le \seminorm{\smash{\omega,\infty}}{a}^{\smash{2n}}$ for all $n\in \NN$, because $\seminorm{\omega,\infty}{\argument}$
  is a $C^*$\=/seminorm on the unital $^*$\=/subalgebra $\Bounded_\omega(\A,\Omega)$ of $\A$ and because
  $\abs{\dupr{\omega}{c}} \le \dupr{\omega}{c^*c}^{\smash{1/2}} \le \seminorm{\omega,\infty}{c}$ for all $c\in \Bounded_\omega(\A,\Omega)$.
  This proves $\smash{\seminorm{\omega,\infty}{a}' \le \seminorm{\omega,\infty}{a}}$.
\end{proof}
If such a boundedness condition is fulfilled,
we can replace symmetry by a different requirement:
\begin{definition}
  An abstract $O^*$\=/algebra $(\A,\Omega)$ is said to be \neu{regular} if every 
  $\omega \in \Omega_\Hermitian$ that is algebraically positive, is positive, i.e. an element of $\Omega^+_\Hermitian$.
\end{definition}
Note that for regular abstract $O^*$-algebras, the set of characters $\Characters(\A,\Omega)$
is just the set of all $\omega\in\Omega$ which are unital $^*$-homomorphisms to $\CC$,
because such homomorphisms are algebraically positive, hence positive.
\begin{definition}
  Let $(\A,\Omega)$ be an abstract $O^*$\=/algebra, then define for every $\omega \in \Omega$
  the seminorm $\seminorm{\omega,\wk}{\argument}$ on $\A$ as 
  \begin{equation*}
    a \mapsto \seminorm{\omega,\wk}{a} := \abs{\dupr{\omega}{a}}\,.
  \end{equation*}
  The locally convex topology on $\A$ defined by all $\seminorm{\omega,\wk}{\argument}$ 
  with $\omega \in \Omega$ is called the \neu{weak topology}.
\end{definition}
\begin{proposition} \label{proposition:regular}
  Let $(\A,\Omega)$ be an abstract $O^*$\=/algebra, then $(\A,\Omega)$ is regular
  if and only if the algebraically positive elements $\A^{++}_\Hermitian$ 
  are weakly dense in the positive elements $\A^+_\Hermitian$.
\end{proposition}
\begin{proof}
  If $\A^{++}_\Hermitian$ is dense in $\A^+_\Hermitian$ and $\omega \in \Omega_\Hermitian$
  algebraically positive,
  then it follows from $\dupr{\omega}{a} \ge 0$ for all $a\in\A^{++}_\Hermitian$ 
  that $\dupr{\omega}{a} \ge 0$ for all $a\in \A^+_\Hermitian$, hence $\omega\in\Omega^+_\Hermitian$.
  
  Conversely, if $\A^{++}_\Hermitian$ is not dense in $\A^+_\Hermitian$, then
  there exist $a\in \A^+_\Hermitian$ and $\omega \in \Omega_\Hermitian$ such that
  $\dupr{\omega}{b} \ge \dupr{\omega}{a} + 1$ holds for all $b\in \A^{++}_\Hermitian$
  by the separation theorem for the closure of the convex set $\A^{++}_\Hermitian$
  from the compact point $a \in \A^+_\Hermitian \backslash \A^{++}_\Hermitian$
  in the real locally convex space $\A_\Hermitian$ carrying the weak topology of $\Omega_\Hermitian$
  (take the quotient with the closure of $\{0\}$ if necessary), again, see e.g. \cite{koethe},
  §20.7 (2). Even more,
  $\dupr{\omega}{b} \ge 0 \ge \dupr{\omega}{a} + 1$ for all $b\in \A^{++}_\Hermitian$
  because $\A^{++}_\Hermitian$ is a cone, so $\omega$ is algebraically positive
  but $\dupr{\omega}{a} < 0$, i.e. $(\A,\Omega)$ is not regular.
\end{proof}
\begin{proposition} \label{proposition:v1}
  Let $(\A,\Omega)$ be a regular commutative abstract $O^*$\=/algebra and $\omega \in \PureStates(\A,\Omega)$,
  then the restriction of $\omega$ to $\Bounded_\omega(\A,\Omega)$ is multiplicative.
\end{proposition}
\begin{proof}
  This is essentially the argument given in \cite{reptheoforposfun}. In our case it is
  sufficient to check that the unital $^*$-subalgebra $\Bounded_\omega(\A,\Omega)$ of $\A$ fulfils 
  the condition of Proposition~\ref{proposition:extremal}, which is clear:
  For every $a\in \Bounded_\omega(\A,\Omega)$ the inequality $a\acts \omega \le \seminorm{\omega,\infty}{a}^2 \,\omega$
  is fulfilled, because 
  \begin{equation*}
    \dupr[\big]{\seminorm{\omega,\infty}{a}^2\,\omega-(a\acts \omega)}{b^*b} 
    =
    \seminorm{\omega,\infty}{a}^2 \dupr{\omega}{b^*b} - \dupr{b\acts \omega}{a^2}
    \ge 0
  \end{equation*}
  holds for all $b\in\A$ and because $(\A,\Omega)$ was assumed to be regular.
\end{proof}
As an immediate consequence of the above Proposition~\ref{proposition:v1} and 
Proposition~\ref{proposition:charactersAreExtreme} we get:
\begin{theorem} \label{theorem:bounded}
  Let $(\A,\Omega)$ be a commutative regular abstract $O^*$\=/algebra.
  If $\Bounded(\A,\Omega)$ is weakly dense in $\A$, then $\PureStates(\A,\Omega) = \Characters(\A,\Omega)$.
\end{theorem}
\begin{proof}
  It is only left to check that an $\omega\in \PureStates(\A,\Omega)$, that
  is multiplicative on $\Bounded(\A,\Omega)$, is also multiplicative on $\A$. 
  This is true because 
  $\A \ni a \mapsto \dupr{\omega}{ba} - \dupr{\omega}{b}\dupr{\omega}{a} = 
  \dupr{\omega\cdot b}{a} - \dupr{\omega}{b}\dupr{\omega}{a} \in \CC$
  is weakly continuous for all $b\in\A$ and vanishes on $\Bounded(\A,\Omega)$ by Lemma~\ref{lemma:var},
  because $\Var_\omega(a) = 0$ for all $a\in \Bounded(\A,\Omega)$.
\end{proof}
\begin{example} \label{example:1}
  Let $\A$ be a locally convex $^*$\=/algebra, i.e. a $^*$\=/algebra that carries a locally convex
  topology that makes the $^*$-involution as well as the operators of left- and 
  right-multiplication on $\A$ with fixed elements of $\A$ continuous.
  Define $P^+_\Hermitian$ as the set of all continuous algebraically positive 
  linear functionals on $\A$ and construct the abstract $O^*$\=/algebra
  $(\A,\Omega)$ like in Proposition~\ref{proposition:construction}. 
  Then $(\A,\Omega)$ is regular, $\Omega^+_\Hermitian = P^+_\Hermitian$ and
  $\A^+_\Hermitian$ is the closure in $\A_\Hermitian$ of $\A^{++}_\Hermitian$ 
  with respect to the topology of $\A$.

  If $\A$ is even an lmc $^*$\=/algebra, i.e. if the topology of $\A$ can be defined
  by an upwards directed set of submultiplicative seminorms, then $\Bounded(\A,\Omega) = \Bounded_\omega(\A,\Omega) = \A$ for all 
  $\omega\in\States(\A,\Omega)$ and thus $\PureStates(\A,\Omega) = \Characters(\A,\Omega)$
  by the above Theorem~\ref{theorem:bounded} if $\A$ is commutative.
\end{example}
\begin{proof}
  By Proposition~\ref{proposition:construction}, $(\A,\Omega)$ is indeed an abstract $O^*$\=/algebra
  with $\Omega^+_\Hermitian \supseteq P^+_\Hermitian$, and $\Omega^+_\Hermitian \subseteq P^+_\Hermitian$
  holds because every $\omega \in \Omega^+_\Hermitian$ is continuous by construction and 
  algebraically positive because $\A^{++}_\Hermitian \subseteq \A^+_\Hermitian$. Given any
  $a\in \A_\Hermitian$ that is not in the closure of $\A^{++}_\Hermitian$, then 
  by the Hahn-Banach theorem, there
  exists a continuous (real) linear functional $\tilde{\omega}\colon \A_\Hermitian \to \RR$
  such that $\dupr{\tilde{\omega}}{b} \ge \dupr{\tilde{\omega}}{a} +1$ holds for all
  $b\in \A^{++}_\Hermitian$ because the closure of $\A^{++}_\Hermitian$ in $\A_\Hermitian$ is convex, 
  (again, see e.g. \cite{koethe}, §20.7 (2)) and even
  $\dupr{\tilde{\omega}}{b} \ge 0 \ge \dupr{\tilde{\omega}}{a} +1$ for all
  $b\in \A^{++}_\Hermitian$ because $\A^{++}_\Hermitian$ is a cone. So $\tilde{\omega}$
  extends to a continuous (complex) linear functional $\omega \in \Omega^+_\Hermitian$
  and $\dupr{\omega}{a} < 0$ shows that $a\notin \A^+_\Hermitian$. 
  Conversely, $\A^+_\Hermitian$ is the intersection of the preimages of the closed interval $[0,\infty[\,$
  under all the continuous $\omega \in \Omega^+_\Hermitian$ and hence is closed.
  Thus the closure of $\A^{++}_\Hermitian$ in $\A_\Hermitian$ with respect to the topology of $\A$
  is $\A^+_\Hermitian$. Therefore $\A^{++}_\Hermitian$ is especially weakly dense in $\A^+_\Hermitian$,
  and $(\A,\Omega)$ is regular by Proposition~\ref{proposition:regular}.
  
  Given $\omega \in \States(\A,\Omega)$, then
  $\dupr{b\acts \omega}{a^*a} \le \dupr{b\acts \omega}{(a^*a)^n}^{1/n}$
  for all $n\in \NN$ and all $a,b\in \A$ by Lemma~\ref{lemma:growth}. If 
  $\A$ is even an lmc $^*$\=/algebra, then there exists a continuous submultiplicative
  seminorm $\seminorm{}{\argument}$ on $\A$ fulfilling
  $\dupr{b\acts \omega}{a^*a} \le C^{1/n}\seminorm{}{b}^{1/n}\seminorm{}{a^*a}\seminorm{}{b^*}^{1/n} \xrightarrow{n\to \infty} \seminorm{}{a^*a}$
  for all $a,b\in \A$ with some $C\in [0,\infty[\,$. So $\seminorm{\omega,\infty}{a} \le \seminorm{}{a^*a}^{1/2} < \infty$ and $\omega$ is
  bounded for all $a\in \A$. This shows $\Bounded(\A,\Omega) = \Bounded_\omega(\A,\Omega) = \A$ for all
  $\omega \in \States(\A,\Omega)$.
  If $\A$ is commutative we can thus apply Theorem~\ref{theorem:bounded}.
\end{proof}
\section{Stieltjes states} \label{sec:unbounded}
In order to treat cases where no a priori boundedness assumptions can be made,
we will have to assume that most algebra elements are at least somehow
dominated by essentially self-adjoint ones:
\begin{definition}
  Let $\A$ be a quasi-ordered $^*$\=/algebra. An element $q\in \A^+_\Hermitian$
  is called \neu{coercive} if there exists an $\epsilon>0$ such that 
  $q\gtrsim \epsilon \Unit$. Let $Q\subseteq \A_\Hermitian$ be a non-empty set
  of pairwise commuting elements and such that $q^2$ is coercive and $\lambda q \in Q$
  as well as $qr \in Q$ hold for all $q,r\in Q$ and all $\lambda\in[1,\infty[\,$;
  such a set will be called \neu{dominant}. Then define
  \begin{equation*}
    Q^\downarrow := \set[\big]{a\in\A}{\forall_{q\in Q} \exists_{r,s\in Q} : a^*q^2 a \lesssim r^2\text{ and } a\,q^2a^*\lesssim s^2}\,.
  \end{equation*}
\end{definition}
Note that this especially implies that for every $a\in Q^\downarrow$ there exists an $r\in Q$
such that $a^*a \lesssim r^2$ holds, and if $aq=qa$ and $aa^*=a^*a$ hold for all $q\in Q$,
especially if $\A$ is commutative, then $a^*a \lesssim r^2$ is even sufficient
for an $a\in \A$ to be in $Q^\downarrow$.
\begin{lemma} \label{lemma:coercive1}
  Let $\A$ be a quasi-ordered $^*$\=/algebra and $q,r \in \A_\Hermitian$ commuting 
  elements with the property that $q^2$ and $r^2$ are coercive.
  Then $\lambda q^2 r^2$ is coercive for all $\lambda \in \,\,]0,\infty[\,$ and there
  exists a $\lambda\in [1,\infty[\,$ such that $q^2+r^2 \lesssim \lambda q^2 r^2$ holds.
\end{lemma}
\begin{proof}
  Let $2\ge \epsilon >0$ be given such that $q^2 \gtrsim \epsilon \Unit$ and
  $r^2 \gtrsim \epsilon \Unit$, then
  \begin{equation*}
    (2/\epsilon) \,q^2 r^2 = q\,(r^2/ \epsilon-\Unit)\,q + r\,(q^2 / \epsilon - \Unit)\,r + q^2 + r^2 \gtrsim q^2 + r^2
  \end{equation*}
  holds. So $q^2+r^2 \lesssim \lambda q^2 r^2$ if one chooses $\lambda := 2/\epsilon\ge 1$ 
  and $\lambda q^2 r^2$ is coercive for all $\lambda \in \,\,]0,\infty[\,$ because
  $\lambda q^2 r^2 \gtrsim (\epsilon\lambda / 2) \,(q^2+r^2) \gtrsim \epsilon^2\lambda \Unit$.
\end{proof}
\begin{proposition} \label{proposition:domalgebra}
  Let $\A$ be a quasi-ordered $^*$\=/algebra and $Q\subseteq \A_\Hermitian$
  dominant, then $Q^\downarrow$ is a unital $^*$\=/subalgebra of $\A$, even
  a quasi-ordered $^*$\=/algebra with the order inherited from $\A$,
  and $Q\subseteq Q^\downarrow$.
\end{proposition}
\begin{proof}
  It is immediately clear that $Q^\downarrow$ is stable under the $^*$-involution and under
  multiplication with scalars and that $\Unit \in Q^\downarrow$. Given $a,b\in Q^\downarrow$
  and $q\in Q$, then there exist $r,s,t\in Q$ such that
  $a^*q^2 a \lesssim r^2$ and $b^*q^2 b \lesssim s^2$ as well as $b^*r^2 b \lesssim t^2$ hold, so
  \begin{equation*}
    (a+b)^*q^2(a+b) 
    \lesssim 
    (a+b)^*q^2(a+b) + (a-b)^*q^2(a-b) 
    =
    2a^*q^2a + 2b^*q^2 b
    \lesssim
    2\,(r^2+s^2) 
    \lesssim 
    2\lambda r^2 s^2
  \end{equation*}
  with sufficiently large $\lambda\in[1,\infty[\,$ by the previous Lemma~\ref{lemma:coercive1},
  and $(ab)^*q^2(ab) = b^*a^*q^2a\,b \lesssim b^*r^2b \lesssim t^2$. Of course,
  there are similar estimates for $a$ and $b$ replaced by $a^*$ and $b^*$,
  and thus $a+b \in Q^\downarrow$ and $ab\in Q^\downarrow$.
  This shows that $Q^\downarrow$ is a unital $^*$\=/subalgebra of $\A$ and it is clear that 
  it is even a quasi-ordered $^*$-algebra with the order inherited from $\A$.
  Finally, $Q\subseteq Q^\downarrow$ is an immediate consequence of the closedness of $Q$ under multiplication
  and its commutativity.
\end{proof}
In the special case of $O^*$\=/algebras, this dominated unital $^*$\=/subalgebra $Q^\downarrow$
has a particularly easy interpretation as a $^*$\=/algebra of continuous adjointable endomorphisms:
\begin{proposition} \label{proposition:dominatedchar}
  Let $\Hilb$ be a Hilbert space, $\Dom\subseteq \Hilb$ a dense linear subspace and 
  $Q\subseteq \Adbar(\Dom)_\Hermitian$ dominant. Then $\set{\seminorm{q^2}{\argument}}{q\in Q}$
  is a cofinal subset of the set of all seminorms on $\Dom$ that are continuous
  with respect to $\tau_{Q^\downarrow}$ like in Definition~\ref{definition:Ostar}.
  Moreover, given $a \in \Adbar(\Dom)$, then $a\in Q^\downarrow$ holds if and only if
  $a$ and $a^*$ are both continuous with respect to $\tau_{Q^\downarrow}$.
\end{proposition}
\begin{proof}
  As $Q\subseteq Q^\downarrow$ by the previous Proposition~\ref{proposition:domalgebra},
  it is clear that all $\seminorm{q^2}{\argument}$ are $\tau_{Q^\downarrow}$\=/continuous. 
  Conversely, the set $\set{\seminorm{a}{\argument}}{a\in (Q^\downarrow)^+_\Hermitian}$
  defines the $\tau_{Q^\downarrow}$-topology and is upwards directed and closed under
  multiplication with non-negative scalars because $(Q^\downarrow)^+_\Hermitian$ is.
  As for all $a \in (Q^\downarrow)^+_\Hermitian$ there exists
  a $q\in Q$ such that $a \le (a+\Unit)^2 \le q^2$, hence 
  $\seminorm{a}{\argument} \le \seminorm{q^2}{\argument}$,
  this shows that $\set{\seminorm{q^2}{\argument}}{q\in Q}$ is a cofinal subset of the 
  $\tau_{Q^\downarrow}$\=/continuous seminorms on $\Dom$.
  
  If $a\in Q^\downarrow$, then $a$ is certainly $\tau_{Q^\downarrow}$\=/continuous, because
  $\seminorm{b}{a\phi} = \seminorm{a^*ba}{\phi}$ holds for all $b\in (Q^\downarrow)^+_\Hermitian$
  and all $\phi\in\Dom$. It also follows that $a^*$ is $\tau_{Q^\downarrow}$\=/continuous
  because $a^*\in Q^\downarrow$ as well.
  Conversely, given an $a\in \Adbar(\Dom)$ such that $a$ is $\tau_{Q^\downarrow}$\=/continuous,
  then for every $q\in Q$ there exists an $r\in Q$ such that
  $\seminorm{q^2}{a\phi} \le \seminorm{r^2}{\phi}$ holds for all $\phi\in \Dom$,
  hence $a^*q^2 a \le r^2$. If $a^*$ is $\tau_{Q^\downarrow}$\=/continuous as well, then
  there also exists an $s\in Q$
  such that $a\,q^2a^* \le s^2$ and we conclude that $a\in Q^\downarrow$.
\end{proof}
If $Q^\downarrow \subseteq \Adbar(\Dom)$ is a closed $O^*$\=/algebra
and every $q^2$ with $q\in Q$ is essentially self-adjoint, then $Q^\downarrow$
is especially well-behaved. Such a $Q^\downarrow$ would
be an example of a \neu{strictly self-adjoint} $O^*$\=/algebra in the language
of \cite{schmuedgen}, Definition~7.3.6:
\begin{lemma} \label{lemma:restrict}
  Let $\Hilb$ be a Hilbert space, $\Dom\subseteq \Hilb$ a dense linear subspace
  and $Q\subseteq \Adbar(\Dom)$ a dominant set with the properties that
  $Q^\downarrow \subseteq \Adbar(\Dom)$ is a closed $O^*$\=/algebra on $\Dom$
  and that $q^2$ is essentially self-adjoint for every $q\in Q$.
  Then $\Dom = \bigcap_{q\in Q} \Dom[(q^2)^\dagger]$ and every bounded operator 
  $B \in \Adbar(\Hilb)$ that fulfils $\skal{B\phi}{q \psi} = \skal{B q \phi}{\psi}$
  and $\skal{B^*\phi}{q\psi} = \skal{B^* q \phi}{\psi}$
  for all $\phi,\psi \in \Dom$ and all $q\in Q$ can be restricted to some
  $b \in Q^\downarrow$ which commutes with all $q\in Q$. As a
  special case, every $q\in Q$ has a bounded inverse $q^{-1} \in Q^\downarrow$.
\end{lemma}
\begin{proof}
  This argument is similar to the proof of Proposition 7.2.2 in \cite{schmuedgen}:
  It is clear that $\Dom \subseteq \bigcap_{q\in Q} \Dom[(q^2)^\dagger]$. 
  Conversely, $\Dom[(q^2)^\dagger] = \Dom[(q^2)^\cl]$ for every $q\in Q$
  because $q^2$ is essentially self-adjoint and 
  $\Dom = \bigcap_{a\in Q^\downarrow} \Dom[a^\cl]$ because $Q^\downarrow$
  is a closed $O^*$-algebra. 
  Given $a\in Q^\downarrow$, then there exists a $q\in Q$ with $q^2 \gtrsim \Unit$ 
  such that $a^*a \lesssim q^2$, hence also $\Unit+a^*a \lesssim \Unit+q^4$ by
  using that $q^4 = q\,(q^2 -\Unit)\,q + q^2 \gtrsim q^2$. So 
  $\seminorm{\Unit+a^*a}{\argument}$ on $\Dom$ is dominated by 
  $\seminorm{\Unit+q^4}{\argument}$, hence $\Dom[a^\cl] \supseteq \Dom[(q^2)^\cl]$ and
  $\Dom = \bigcap_{a\in Q^\downarrow} \Dom[a^\cl] \supseteq 
  \bigcap_{q\in Q} \Dom[(q^2)^\cl] = \bigcap_{q\in Q} \Dom[(q^2)^\dagger]$.
  
  Now let $B \in \Adbar(\Hilb)$ be given that fulfils $\skal{B\phi}{q \psi} = \skal{B q \phi}{\psi}$
  and $\skal{B^*\phi}{q\psi} = \skal{B^* q \phi}{\psi}$
  for all $\phi,\psi \in \Dom$ and all $q\in Q$. Then 
  $\Dom \ni \psi \mapsto \skal{B \phi}{q^2 \psi} = \skal{Bq^2\phi}{\psi} \in \CC$
  is for all $\phi \in \Dom$ and all $q\in Q$ a $\seminorm{}{\argument}$\=/continuous linear functional,
  hence $B\phi \in \Dom[(q^2)^\dagger]$ and even $B\phi \in \bigcap_{q\in Q} \Dom[(q^2)^\dagger] = \Dom$.
  So $B$ can be restricted to a linear function $b \colon \Dom \to \Dom$.
  The same argument applied to $B^*$ shows that $b$ is adjointable, i.e. $b\in \Adbar(\Dom)$.
  Clearly, $b$ and $b^*$ commute with all $q\in Q$ and are bounded,
  from which it follows that $b,b^*\in Q^\downarrow$.
  
  For every $q\in Q$ the coercive essentially self-adjoint $q^2$ is injective
  and its image is dense in $\Hilb$ with respect to $\seminorm{}{\argument}$.
  Even more, for every $r \in Q$ the image of $q^2$ is dense in $\Dom$
  with respect to $\seminorm{\Unit+r^2}{\argument}$:
  Let $\psi \in \Dom$ and $\epsilon > 0$ be given.
  As $r^2$ is coercive, the norm $\seminorm{\Unit+r^2}{\argument}$ is equivalent to the norm
  $\seminorm{r^2}{\argument}$, and we can assume without loss of generality that $r^2 \ge \Unit$, hence
  $r^4 = r\,(r^2-\Unit)\,r + r^2 \ge r^2$.
  Being coercive and essentially self-adjoint, $r^2q^2$ has dense image in $\Hilb$
  with respect to $\seminorm{}{\argument}$, and so there exists a $\phi\in \Dom$
  such that $\seminorm{}{r^2 q^2\phi - r^2 \psi} \le \epsilon$ holds, hence 
  $\seminorm{r^2}{q^2\phi - \psi} \le \seminorm{r^4}{q^2\phi - \psi} = \seminorm{}{r^2 q^2\phi - r^2 \psi} \le \epsilon$.
  
  As $q^2$ is coercive, injective and has dense image in $\Hilb$,
  it follows that $q^2$ has a bounded Hermitian (left-)inverse $B \in \Adbar(\Hilb)_\Hermitian$,
  which fulfils 
  \begin{equation*}
    \skal{B \phi}{r q^2 \psi} = \skal{\phi}{B q^2 r \psi} = \skal{\phi}{ r \psi} 
    = \skal{r \phi}{\psi} = \skal{B r \phi}{q^2\psi}
  \end{equation*}
  for all $\phi, \psi \in \Dom$ and all $r\in Q$, hence
  $\skal{B \phi}{r \psi} = \skal{B r \phi}{\psi}$ for all $\phi, \psi \in \Dom$ and all $r\in Q$
  by using that the image of $q^2$ is dense in $\Dom$ with respect to $\seminorm{\Unit+r^2}{\argument}$.
  So $B$ restricts to a left inverse $b \in Q^\downarrow$ of $q^2$, which 
  commutes with $q^2$ and therefore is also a right inverse. Then $q^{-1} := qb \in Q^\downarrow$
  is the inverse of $q$.
\end{proof}
In order to guarantee that all squares of elements of such a dominant set $Q$ are 
essentially self-adjoint, a variant of Nelson's theorem will be helpful:
\begin{definition} \label{definition:stieltjesstate}
  Let $(\A,\Omega)$ be an abstract $O^*$\=/algebra, $\omega\in \States(\A,\Omega)$ 
  and $a\in \A^+_\Hermitian$,
  then $\omega$ is said to be a \neu{Stieltjes state for $a$} if
  \begin{equation*}
    \dupr{b\acts\omega}{a} = 0 \quad\quad\text{or}\quad\quad \sum_{n=1}^\infty \dupr{b\acts\omega}{a^n}^{-\frac{1}{2n}} = \infty
  \end{equation*}
  holds for all $b\in\A$ with $\dupr{\omega}{b^*b} = 1$ (in analogy to the notion
  of Stieltjes vectors, see e.g. \cite{masson.mcclary:1972}).
\end{definition}
Note that Lemma~\ref{lemma:growth} assures that in the above definition,
$\dupr{b\acts\omega}{a} \neq 0$ implies that $\dupr{b\acts\omega}{a^{n}} > 0$
for all $n\in \NN$ and that
$\smash{n\mapsto \dupr{b\acts\omega}{a^{n}}^{-1/(2n)}}$ is non-increasing. If
$\omega$ is a bounded state for $a$, then either $\seminorm{\omega,\infty}{a} = 0$,
in which case $\dupr{b\acts\omega}{a} = 0$, or 
$\dupr{b\acts\omega}{a^n}^{-1/(2n)} \ge \dupr{b\acts\omega}{a^{2n}}^{-1/(4n)}\ge \seminorm{\omega,\infty}{a}^{-1/2}$ for
all $n\in \NN$, so every bounded state for $a$ is also a Stieltjes state. However,
the notion of a Stieltjes state is much less restrictive. For example, 
if $\dupr{b\acts\omega}{a^{n}} \le C_b(2n)^{2n}$ holds for all $n\in \NN$
with an arbitrary $C_b\in[0,\infty[\,$,
which may depend on $b$, then $\omega$ is a Stieltjes state for $a$.
Like for bounded states, this growth condition depends only on $\omega$ and powers of $a$ if $\A$ is commutative:
\begin{proposition}
  Let $(\A,\Omega)$ be a commutative abstract $O^*$\=/algebra, $\omega \in \States(\A,\Omega)$ and
  $a\in\A^+_\Hermitian$. Then $\omega$ is a Stieltjes state for $a$ if and only if
  \begin{equation*}
    \dupr{\omega}{a} = 0 \quad\quad\text{or}\quad\quad \sum_{n=1}^\infty \dupr{\omega}{a^{n}}^{-\frac{1}{2n}} = \infty
  \end{equation*}
  holds.
\end{proposition}
\begin{proof}
  This is clearly necessary, but also sufficient:
  If $\dupr{\omega}{a} = 0$, then $\Var_\omega(a)=0$ by 
  Lemma~\ref{lemma:growth}, so
  $\dupr{b\acts \omega}{a} = \dupr{\omega}{b^*b\,a} = \dupr{\omega}{b^*b}\dupr{\omega}{a} = 0$ for all
  $b\in \A$ by Lemma~\ref{lemma:var}
  and $\omega$ is a Stieltjes state. Otherwise $\dupr{\omega}{a^{n}} > 0$ and
  $\dupr{b\acts\omega}{a^{n}}^{1/(2n)} \le \dupr{\omega}{(b^*b)^2}^{1/(4n)}\dupr{\omega}{a^{2n}}^{1/(4n)}$
  for all $b\in \A$ and all $n\in \NN$ due to the Cauchy Schwarz inequality,
  so either $\dupr{b\acts \omega}{a} = 0$ or
  \begin{equation*}
    \sum_{n=1}^\infty \dupr{b\acts\omega}{a^{n}}^{-\frac{1}{2n}}
    \ge
    \sum_{n=1}^\infty \dupr{\omega}{(b^*b)^2}^{-\frac{1}{4n}} \, \dupr{\omega}{a^{2n}}^{-\frac{1}{4n}}
    \ge
    \frac{\dupr{\omega}{(b^*b)^2}^{-\frac{1}{4}}}{2} \sum_{n=2}^\infty \dupr{\omega}{a^{n}}^{-\frac{1}{2n}}
    =
    \infty
  \end{equation*}
  for all $b\in\A$ with $\dupr{\omega}{b^*b} = 1$ by using that $1 = \dupr{\omega}{b^*b}^2 \le \dupr{\omega}{(b^*b)^2}$ and that
  $\dupr{\omega}{a^{2n}}^{1/(4n)} \le \dupr{\omega}{a^{2n+1}}^{1/(4n+2)}$ for all $n\in \NN$ by
  Lemma~\ref{lemma:growth} again.
\end{proof}
If the states of an abstract $O^*$-algebra are Stieltjes states for sufficiently 
many coercive algebra elements, then we will be able to show that the closed GNS
representations fulfil the conditions of Lemma~\ref{lemma:restrict}:
\begin{lemma} \label{lemma:dominantconstruction}
  Let $\A$ be a quasi-ordered $^*$\=/algebra and $Q'\subseteq \A^+_\Hermitian$
  a non-empty set of coercive and pairwise commuting elements. Then
  \begin{equation*}
    Q := \set[\Big]{\lambda\prod\nolimits_{n=1}^N q_n'}{\lambda\in[1,\infty[\,;\,N\in \NN;\,q_1',\dots,q_N'\in Q'}
  \end{equation*}
  is dominant and $Q'\subseteq Q$.
\end{lemma}
\begin{proof}
  If $q'\in \A^+_\Hermitian$ is coercive, then $q'\gtrsim \epsilon \Unit$ and thus
  $(q')^2 = (q'-\epsilon \Unit)^2+ 2 \epsilon \,(q'-\epsilon \Unit) + \epsilon^2 \Unit \gtrsim \epsilon^2 \Unit$
  hold for some $\epsilon > 0$. So $(q')^2$ is also coercive. From Lemma~\ref{lemma:coercive1} it now follows that $q^2$
  is coercive for all $q\in Q$. As $Q$ is closed under the multiplications
  and pairwise commuting by construction, $Q$ is dominant.
  Finally, $Q'\subseteq Q$ is obvious.
\end{proof}
\begin{lemma} \label{lemma:stieltjesvector}
  Let $\Hilb$ be a Hilbert space, $\Dom\subseteq \Hilb$ a dense linear subspace
  and $q\in \Adbar(\Dom)^+_\Hermitian$ coercive. If $S\subseteq \Dom$
  is a linear subspace of $\Dom$ consisting of Stieltjes vectors for $q$, i.e. if
  \begin{equation*}
    \sum_{n=1}^\infty \frac{1}{\seminorm{}{q^n\phi}^{1/(2n)}} = \infty
  \end{equation*}
  holds for all $\phi$ in $S\backslash\{0\}$, and if $q$ can be restricted to an endomorphism
  $q|_S\colon S \to S$, then every $\psi \in \Hilb$ that is orthogonal
  on all $q\phi$ with $\phi \in S$ is also orthogonal on all $\phi\in S$,
  so $\set{q\phi}{\phi\in S}^\bot = S^\bot$.
\end{lemma}
\begin{proof}
  This is a variant of Nelson's criterium for self-adjoint operators, which
  was essentially proven in \cite{masson.mcclary:1972}, Lemma~6. 
  For the convenience of the reader, the proof and some adaptations will be outlined:
  
  Let $\psi \in \Hilb$ be given such that $\skal{\psi}{q\phi} = 0$ for all 
  $\phi \in S$, then also $\skal{\psi}{(-q)^n\phi} = 0$ for all $\phi \in S$.
  For a fixed $\phi \in S\backslash\{0\}$, let $\mathscr{H}_\phi \subseteq S$ be the linear span of the
  $(-q)^n \phi$ for all $n\in \NN$. If $\mathscr{H}_\phi$ has finite dimension,
  then it follows from basic linear algebra that the coercive $q$ has an inverse
  $\tilde{q}^{-1}$ on $\mathscr{H}_\phi$ and thus 
  $\skal{\psi}{\phi} = \skal{\psi}{q\, \tilde{q}^{-1}\phi} = 0$. Otherwise, the
  set $\set{(-q)^n \phi}{n\in \NN_0}$ is a basis of $\mathscr{H}_\phi$ and the
  usual Gram-Schmidt orthogonalisation procedure applied to this basis
  yields an orthogonal basis $\set{f_n}{n\in \NN_0}$ starting with $f_0 = \phi$,
  from which one can construct an orthonormal basis $\set{e_n = f_n / \seminorm{}{f_n}}{n\in \NN_0}$ of 
  $\mathscr{H}_\phi$. The condition $\skal{\psi}{(-q)\, e_n} = 0$ for all $n\in \NN$
  then is equivalent to the infinite system of linear equations 
  \begin{align*}
    \cc{\psi}_0 a_0 + \cc{\psi}_1 b_0 &= \lambda \cc{\psi}_0 \\
    \text{and}\quad\quad
    \cc{\psi}_{i-1} b_{i-1} + \cc{\psi}_i a_i + \cc{\psi}_{i+1} b_i &= \lambda \cc{\psi}_i
    \quad\quad\text{for all }i\in \NN
  \end{align*}
  for $\lambda = 0$ with
  $\cc{\psi}_i = \skal{\psi}{e_i} \in \CC$, $a_i = \skal{e_i}{(-q)\,e_i} \in \RR$
  as well as $b_i = \skal{e_{i+1}}{(-q)\, e_i} = \seminorm{}{f_{i+1}} / \seminorm{}{f_i}>0$ for all $i\in \NN_0$.
  
  The above system of linear equations for the sequence $(\cc{\psi}_i)_{i\in \NN_0}$ has,
  for every choice of $\lambda,\cc{\psi}_0 \in \CC$, a unique
  solution, which can easily be constructed recursively. If $\cc{\psi}_0 \neq 0$, then we can assume
  without loss of generality that $\cc{\psi}_0 = 1$. Denote for every $\lambda\in \CC$
  by $p_n(\lambda) := \cc{\psi}_n$ the solution starting with $p_0 = \cc{\psi}_0 = 1$.
  Then one can show (\cite{masson.mcclary:1972}, proof of Lemma~6) that
  \begin{equation*}
    \infty
    =
    \sum_{n=1}^\infty \frac{1}{\seminorm{}{q^n \phi}^{1/(2n)}} 
    \le
    \E \,\bigg(\frac{1-c}{\lambda-\mu}\bigg)^{\frac{1}{2}}\,\bigg(\sum_{n=0}^\infty \abs{p_n(\lambda)}^2 \bigg)^{\frac{1}{2}}
  \end{equation*}
  holds with some $c\in[0,1]$ and Euler's constant $\E$
  for all $\mu,\lambda \in \RR$ which fulfil the conditions that
  $\mu < \lambda$ and that there
  exists an $\epsilon > 0$ with the property that $q-(\epsilon-\mu) \Unit \ge 0$
  (i.e. $-q$ is semibounded from above by $\mu-\epsilon$).
  As $q$ is coercive, this especially applies to $\lambda = 0$, which
  yields the contradiction 
  $\infty \le \big(\sum_{n=0}^\infty \abs{p_n(0)}^2\big)^{1/2} = \seminorm{}{\psi}$,
  so $0 = \cc{\psi}_0 = \skal{\psi}{e_0} =  \skal{\psi}{\phi} / \seminorm{}{\phi}$.
  Thus $\set{q\phi}{\phi\in S}^\bot \subseteq S^\bot$ and we conclude that
  $\set{q\phi}{\phi\in S}^\bot = S^\bot$ because $\set{q\phi}{\phi\in S} \subseteq S$
  by assumption.
\end{proof}
As a coercive $q\in \Adbar(\Dom)^+_\Hermitian$ is essentially self-adjoint if and
only if the orthogonal complement of its image in the surrounding Hilbert space 
$\Hilb$ is $\{0\}$,
the above immediately shows that $q$ is essentially self-adjoint if the orthogonal 
complement of $S$ in $\Hilb$ is $\{0\}$. Even more, this argument also applies to 
some products:
\begin{proposition} \label{proposition:essselfad}
  Let $\Hilb$ be a Hilbert space, $\Dom\subseteq \Hilb$ a dense linear subspace
  and $Q' \subseteq \Adbar(\Dom)^+_\Hermitian$ a non-empty set of coercive and pairwise 
  commuting elements. Moreover, assume that every $\phi \in \Dom$ is a 
  Stieltjes vector for every $q' \in Q'$ like in the previous 
  Lemma~\ref{lemma:stieltjesvector} and construct the dominant set $Q$ out of
  $Q'$ like in Lemma~\ref{lemma:dominantconstruction}. Then $q^2$ is essentially
  self-adjoint for every $q\in Q$.
\end{proposition}
\begin{proof}
  Let $q\in Q$ be given, then $q^2$ is coercive because $Q$ is dominant by 
  Lemma~\ref{lemma:dominantconstruction}, so it is sufficient to show that 
  the image of $\Dom$ under $q^2$ is dense in $\Hilb$. By construction of $Q$ it
  is thus sufficient to prove by induction, that for all $N\in \NN$ and all 
  $q_1',\dots,q_N' \in Q'$ the image of $\Dom$ under $q_1' \cdots q_N'$ is dense in $\Hilb$
  (the effect of a multiplication with a scalar $\lambda\in[1,\infty[\,$ is trivial):
  
  If $N=1$, then this is simply the previous Lemma~\ref{lemma:stieltjesvector}
  with $S=\Dom$. So assume that it has been shown for one $N\in \NN$ and all 
  $q_1',\dots,q_N' \in Q'$ that the image of $\Dom$ under $q_1' \cdots q_N'$
  is dense in $\Hilb$ and let $q_1', \dots, q_{N+1}' \in Q'$ be given.
  Define $S$ as the image of $\Dom$ under $q_1'\cdots q_{N}'$, then $S\subseteq \Dom$,
  the orthogonal complement of $S$ in $\Hilb$ is $\{0\}$ by assumption,
  and $q_{N+1}'q_1'\cdots q_{N}' \phi = q_1'\cdots q_{N}'q_{N+1}'\phi \in S$
  holds for all $\phi\in \Dom$, i.e. $q_{N+1}'$ can be restricted to an endomorphism
  of $S$. The previous Lemma~\ref{lemma:stieltjesvector} then applies to $q'_{N+1}$
  and shows that the orthogonal complement of 
  $\set{q'_{N+1}\phi}{\phi\in S} = \set{q'_{1}\dots q'_{N+1}\phi}{\phi\in \Dom}$
  in $\Hilb$ is $\{0\}$.
\end{proof}
As a corollary, we get a construction of strictly self-adjoint $O^*$\=/algebras
out of GNS representations of abstract $O^*$\=/algebras:
\begin{corollary} \label{corollary:main}
  Let $(\A,\Omega)$ be an abstract $O^*$\=/algebra and $Q'\subseteq \A^+_\Hermitian$
  a non-empty set of coercive and pairwise commuting elements and such that every
  $\omega \in \States(\A,\Omega)$ is a Stieltjes state for all $q'\in Q'$.
  Construct the dominant set $Q$ out of $Q'$ like in 
  Lemma~\ref{lemma:dominantconstruction}.
  For $\omega \in \States(\A,\Omega)$, let $\pi_\omega^\cl \colon Q^\downarrow \to \Adbar(\Dom[\omega]^\cl)$
  be the closed GNS representation of $Q^\downarrow$ associated to $\omega$
  like in Definition~\ref{definition:gns}, then $\pi_\omega^\cl(Q^\downarrow) \subseteq \pi_\omega^\cl(Q)^\downarrow$
  and $\pi_\omega^\cl(Q)$ fulfils the conditions of Lemma~\ref{lemma:restrict},
  i.e. $\pi_\omega^\cl(Q)$ is a dominant set with the property that $\pi_\omega^\cl(Q)^\downarrow$ is a closed $O^*$\=/algebra
  on $\Dom[\omega]^\cl$
  and $\pi_\omega^\cl(q)^2$ is essentially self-adjoint for every $\pi_\omega^\cl(q) \in \pi_\omega^\cl(Q)$.
\end{corollary}
\begin{proof}
  If $a\in (Q^\downarrow)^+_\Hermitian$, then $\pi_\omega^\cl(a) \in \Adbar(\Dom[\omega]^\cl)^+_\Hermitian$
  by Lemma~\ref{lemma:gnspositive}. From this it follows that $\pi_\omega^\cl(Q)$ is a dominant set and that
  $\pi_\omega^\cl(Q^\downarrow) \subseteq \pi_\omega^\cl(Q)^\downarrow$.
  Moreover, $\pi_\omega^\cl(Q)^\downarrow$ is a closed $O^*$\=/algebra
  on $\Dom[\omega]^\cl$ because $\pi_\omega^\cl(Q^\downarrow)$ is a closed $O^*$-algebra by
  construction and because the locally convex topologies 
  $\tau_{\pi_\omega^\cl(Q)^\downarrow}$ and $\tau_{\pi_\omega^\cl(Q^\downarrow)}$
  both are induced by the system of seminorms $\set{\seminorm{\pi_\omega^\cl(q^2)}{\argument}}{q\in Q}$
  (see Proposition~\ref{proposition:dominatedchar}).
  Furthermore, $\pi_\omega^\cl(q)^2$ is essentially self-adjoint for every $\pi_\omega^\cl(q) \in \pi_\omega^\cl(Q)$,
  because its restriction to the ordinary GNS representation space $\pi_\omega(q)^2\colon \Dom[\omega] \to \Dom[\omega]$
  is already essentially self-adjoint by the previous Proposition~\ref{proposition:essselfad}.
  Indeed, every $[b]_\omega \in \Dom[\omega]$ is a Stieltjes vector for every $\pi_\omega(q')$ with $q'\in Q'$
  because 
  \begin{equation*}
    \sum_{n=1}^\infty \seminorm[\big]{\omega}{(q')^n [b]_\omega}^{-\frac{1}{2n}}
    =
    \sum_{n=1}^\infty \dupr{b\acts \omega}{(q')^{2n}}^{-\frac{1}{4n}}
    \ge
    \frac{1}{2} \sum_{n=2}^\infty \dupr{b\acts \omega}{(q')^{n}}^{-\frac{1}{2n}}
    =
    \infty
  \end{equation*}
  holds for all $[b]_\omega \neq 0$ by using again that 
  $\dupr{b\acts \omega}{(q')^{2n}}^{1/(4n)} \le \dupr{b\acts \omega}{(q')^{2n+1}}^{1/(4n+2)}$
  for all $n\in\NN$ due to Lemma~\ref{lemma:growth}.
\end{proof}
Now we have all the prerequisits to prove the main results of this section:
\begin{definition}
  Let $(\A,\Omega)$ be an abstract $O^*$\=/algebra, then define for every $\omega \in \Omega^+_\Hermitian$
  the seminorm $\seminorm{\omega,\st}{\argument}$ on $\A$ as 
  \begin{equation*}
    a \mapsto \seminorm{\omega,\st}{a} := \sqrt{\dupr{\omega}{a^*a}}\,.
  \end{equation*}
  The locally convex topology on $\A$ defined by all $\seminorm{\omega,\st}{\argument}$ 
  with $\omega \in \Omega^+_\Hermitian$ is called the \neu{strong topology}.
\end{definition}
\begin{definition}
  An abstract $O^*$\=/algebra $(\A,\Omega)$ is said to have \neu{sufficiently many Stieltjes states}
  if there exists a
  set $Q'\subseteq \A^+_\Hermitian$ of coercive and pairwise commuting elements with the
  following properties:
  \begin{itemize}
    \item Every $\omega\in \States(\A,\Omega)$ is a Stieltjes state for every $q'\in Q'$.
    \item $(Q^\downarrow)^+_\Hermitian$ is strongly dense in $\A^+_\Hermitian$,
      where $Q \subseteq \A_\Hermitian$
      is constructed out of $Q'$ like in Lemma~\ref{lemma:dominantconstruction}.
  \end{itemize}
\end{definition}
\begin{definition}
  An abstract $O^*$\=/algebra $(\A,\Omega)$ is called \neu{downwards closed} if
  every $\rho\in \A^{*,+}_\Hermitian$, for which there exists an $\omega \in \Omega^+_\Hermitian$
  such that $\omega-\rho \in \A^{*,+}_\Hermitian$, is an element of $\Omega^+_\Hermitian$.
\end{definition}
\begin{proposition} \label{proposition:hyperregular}
  Let $(\A,\Omega)$ be a downwards closed abstract $O^*$\=/algebra with sufficiently 
  many Stieltjes states, $\omega \in \Omega^+_\Hermitian$ and let $\rho\in \A^*_\Hermitian$
  be algebraically positive and such that $\omega-\rho$ is also algebraically positive.
  Then $\rho \in \Omega^+_\Hermitian$ and $0 \le \rho \le \omega$.
\end{proposition}
\begin{proof}
  Given $\rho \in \A^*_\Hermitian$ and $\omega \in \Omega^+_\Hermitian$
  such that $\rho$ and $\omega-\rho$ are algebraically positive, then it is sufficient
  to show that $0 \le \dupr{\rho}{c}$ and $0 \le \dupr{\omega-\rho}{c}$ hold for all $c\in \A^+_\Hermitian$,
  because then $\rho\in \Omega^+_\Hermitian$ and $\omega-\rho \in \Omega^+_\Hermitian$
  as $(\A,\Omega)$ is downwards closed.

  Let $Q'\subseteq \A^+_\Hermitian$ be a non-empty set of coercive and pairwise commuting elements such that
  every $\omega \in \States(\A,\Omega)$ is a Stieltjes state for every $q'\in Q'$
  and such that $Q^\downarrow$ is strongly dense in $\A$, where $Q$ is the dominant
  set constructed out of $Q'$ like in Lemma~\ref{lemma:dominantconstruction}.
  Such a set $Q'$ exists because $(\A,\Omega)$ has sufficiently many Stieltjes states.
  Let $\pi_\omega \colon Q^\downarrow \to \Adbar(\Dom[\omega])$
  be the GNS representation of $Q^\downarrow$ associated to $\omega$
  and $\pi_\omega^\cl \colon Q^\downarrow \to \Adbar(\Dom[\omega]^\cl)$
  the corresponding closed GNS representation, and let $\Hilb_\omega$ be the 
  surrounding Hilbert space. Then
  $(\Dom[\omega])^2 \ni ([a]_\omega,[b]_\omega) \mapsto \hat{\rho}([a]_\omega,[b]_\omega) := \dupr{\rho}{a^*b} \in \CC$
  is a well-defined and bounded sesquilinear form: This is due to the observation that
  $0 \le \dupr{\rho}{a^*a} \le \dupr{\omega}{a^*a} = \seminorm[\big]{\omega}{[a]_\omega}^2$
  holds for all $a\in \A$, and especially
  if $a\in \Gelfand_\omega$, then $\dupr{\rho}{a^*b} = \dupr{\rho}{b^*a} = 0$ for all $b\in \A$
  by the Cauchy Schwarz inequality.
  
  As $\Dom[\omega]$ is dense in $\Hilb_\omega$ with respect to $\seminorm{\omega}{\argument}$, 
  the form $\hat{\rho}$ extends to a bounded sesquilinear form on $\Hilb_\omega$, 
  which is clearly Hermitian and positive,
  and so by the Lax-Milgram theorem there exists an $R\in \Adbar(\Hilb_\omega)^+_\Hermitian$
  such that $\dupr{\rho}{a^*b} = \hat{\rho}([a]_\omega,[b]_\omega) = \skal{[a]_\omega}{ R \,[b]_\omega}_\omega$
  holds for all $a,b\in Q^\downarrow$. This especially implies that
  \begin{equation*}
    \skal{[a]_\omega}{ R\, \pi_\omega(c)\,[b]_\omega}_\omega
    =
    \dupr{\rho}{a^*c\,b}
    =
    \dupr{\rho}{(c^* a)^*b}
    = 
    \skal{\pi_\omega(c)^* [a]_\omega}{ R \,[b]_\omega}_\omega
  \end{equation*}
  holds for all $a,b,c\in Q^\downarrow$,
  hence $\skal{\phi}{ R \,\pi_\omega^\cl(c)\,\psi}_\omega = \skal{\pi_\omega^\cl(c)^* \phi}{ R \,\psi}_\omega$
  for all $\phi,\psi \in \Dom[\omega]^\cl$ and all $c\in Q^\downarrow$
  because $\Dom[\omega]$ is $\tau_{\pi_\omega(Q^\downarrow)}$-dense
  in $\Dom[\omega]^\cl$ and because both sides are $\tau_{\pi_\omega(Q^\downarrow)}$\=/continuous
  in $\phi$ and $\psi$.
  Now let $\sqrt{R}\in\Adbar(\Hilb_\omega)^+_\Hermitian$ be the square root of $R$,
  i.e. $\smash{\sqrt{R}^2} = R$. Recall that $\sqrt{R}$ can be constructed as a $\seminorm{\omega}{\argument}$-limit 
  of polynomials of $R$, hence $\skal{\phi}{ \sqrt{R}\, \pi_\omega^\cl(c)\,\psi}_\omega =
  \skal{\pi_\omega^\cl(c)^* \phi}{ \sqrt{R}\, \psi}_\omega$ holds for all 
  $\phi,\psi \in \Dom[\omega]^\cl$ and all $c\in Q^\downarrow$,
  especially for $c=q^2$ with $q\in Q$. By Corollary~\ref{corollary:main},
  Lemma~\ref{lemma:restrict} can be applied and $\sqrt{R}$ can be restricted
  to a $\sqrt{r} \in \pi_\omega^\cl(Q)^\downarrow \subseteq \Adbar(\Dom[\omega]^\cl)$. 
  This yields a new representation of $\rho$ as
  \begin{equation*}
    \dupr{\rho}{c}
    =
    \skal{[\Unit]_\omega }{R \,\pi_\omega^\cl(c)\,[\Unit]_\omega}_\omega
    =
    \skal{\sqrt{r}\,[\Unit]_\omega }{\pi_\omega^\cl(c)\,\sqrt{r}\,[\Unit]_\omega}_\omega    
    =
    \skal{\xi }{\pi_\omega^\cl(c)\,\xi}_\omega
    \ge
    0
  \end{equation*}
  for all $c\in Q^\downarrow$ with $\xi := \sqrt{r}\,[\Unit]_\omega \in \Dom[\omega]^\cl$.
  From Lemma~\ref{lemma:gnspositive} it now follows that $\dupr{\rho}{c} \ge 0$ 
  for all $c\in (Q^\downarrow)^+_\Hermitian$. Moreover, 
  $\abs{\dupr{\rho}{a}} \le \dupr{\rho}{\Unit}^{1/2} \dupr{\rho}{a^*a}^{1/2} \le
  \dupr{\rho}{\Unit}^{1/2} \seminorm{\omega,\st}{a}$ holds for all
  $a\in \A$, so $\rho$ is strongly continuous. As $(Q^\downarrow)^+_\Hermitian$
  is strongly dense in $\A^+_\Hermitian$ by assumption,
  this implies that $\dupr{\rho}{c} \ge 0$ for all $c\in \A^+_\Hermitian$.
  Finally, note that $\rho' := \omega-\rho$ also fulfils the condition that 
  $\rho'$ and $\omega-\rho' = \rho$
  are algebraically positive, so the above also shows that 
  $\dupr{\omega-\rho}{c} = \dupr{\rho'}{c} \ge 0$ for all $c\in \A^+_\Hermitian$.
\end{proof}
\begin{corollary} \label{corollary:regularsufcond}
  Let $(\A,\Omega)$ be a downwards closed abstract $O^*$\=/algebra with sufficiently 
  many Stieltjes states, then $(\A,\Omega)$ is regular.
\end{corollary}
\begin{proof}
  Given an algebraically positive $\rho \in \Omega_\Hermitian$, then there exists
  an $\omega \in \Omega^+_\Hermitian$ such that $\omega-\rho \ge 0$, because $\Omega_\Hermitian$
  is the (real) linear span of $\Omega^+_\Hermitian$. As $\omega-\rho$ is especially algebraically
  positive, the previous Proposition~\ref{proposition:hyperregular} applies and $\rho\in\Omega^+_\Hermitian$.
\end{proof}
\begin{theorem} \label{theorem:unbounded}
  Let $(\A,\Omega)$ be a commutative and downwards closed abstract $O^*$\=/algebra
  with sufficiently many Stieltjes states, then $\PureStates(\A,\Omega) = \Characters(\A,\Omega)$.
\end{theorem}
\begin{proof}
  As $\PureStates(\A,\Omega) \supseteq \Characters(\A,\Omega)$ by Proposition~\ref{proposition:charactersAreExtreme},
  it only remains to show
  that every pure state $\omega$ of $(\A,\Omega)$ is multiplicative. 
  By assumption, there exists
  a non-empty set $Q'\subseteq \A^+_\Hermitian$ of coercive and pairwise commuting elements such that
  every $\omega \in \States(\A,\Omega)$ is a Stieltjes state for every $q'\in Q'$
  and such that $(Q^\downarrow)^+_\Hermitian$ is strongly dense in 
  $\A^+_\Hermitian$, where $Q$ is the dominant
  set constructed out of $Q'$ like in Lemma~\ref{lemma:dominantconstruction}.
  Note that this implies that $Q^\downarrow$ is strongly dense in $\A$ because
  $\A$ is the linear span of $\A^+_\Hermitian$.
  
  Let $\omega\in\PureStates(\A,\Omega)$ be given and let 
  $\pi_\omega^\cl \colon Q^\downarrow \to \Adbar(\Dom[\omega]^\cl)$
  be the closed GNS representation associated to $\omega$.
  Then Corollary~\ref{corollary:main} shows that
  Lemma~\ref{lemma:restrict} applies to the dominant set
  $\pi^\cl_\omega(Q) \subseteq \Adbar(\Dom[\omega]^\cl)$ and there exists an inverse
  $\pi^\cl_\omega(q)^{-1} \in \pi^\cl_\omega(Q)^\downarrow \subseteq \Adbar(\Dom[\omega]^\cl)$
  of $\pi^\cl_\omega(q)$ for every $q\in Q$. 
  
  Let $a\in Q^\downarrow$ be given, then there exists a $q\in Q$ such that 
  $\Unit+a^*a \lesssim q^2$ holds, so especially 
  $b^*b = b^*q^2 b - b^*(q^2-\Unit)\,b \lesssim  b^*q^2b = qb^*bq$ for all $b\in \A$.
  Define the linear functional $\tilde{\rho}_q \colon Q^\downarrow \to \CC$, 
  \begin{equation*}
    b
    \mapsto 
    \dupr{\tilde{\rho}_q}{b} 
    :=
    \skal[\big]{
      \pi^\cl_\omega(q)^{-1} [\Unit]_\omega
    }{
      \pi^\cl_\omega(b) \, \pi^\cl_\omega(q)^{-1} [\Unit]_\omega
    }\,,
  \end{equation*}
  then $\dupr{\tilde{\rho}_q}{b} \ge 0$ for all $b\in (Q^\downarrow)^+_\Hermitian$
  and $\dupr{\tilde{\rho}_q}{b^*b} \le \dupr{\tilde{\rho}_q}{qb^*bq} = \dupr{\omega}{b^*b}$
  for all $b\in Q^\downarrow$ by Lemma~\ref{lemma:gnspositive}. 
  
  Now $\big(Q^\downarrow, \Omega|_{Q^\downarrow}\big)$ with $\Omega|_{Q^\downarrow}$
  the restriction of all funtionals in $\Omega$ to $Q^\downarrow$ is again
  an abstract $O^*$-algebra with sufficiently many Stieltjes states and even downwards
  closed, as every positive linear functional on $Q^\downarrow$ dominated by a functional
  in $\Omega^+_\Hermitian$ is strongly continuous and thus extends to a functional
  in $\Omega^+_\Hermitian$. So the previous Proposition~\ref{proposition:hyperregular}
  shows that $\tilde{\rho}_q \in (\Omega|_{Q^\downarrow})^+_\Hermitian$ and
  $\omega|_{Q^\downarrow}-\tilde{\rho}_q\in (\Omega|_{Q^\downarrow})^+_\Hermitian$,
  so by the same argument as before, $\tilde{\rho}_q$ is the restriction of some 
  $\rho_q \in \Omega^+_\Hermitian$ and $\rho_q \le \omega$, hence $\rho_q = \dupr{\rho_q}{\Unit} \, \omega$
  by Lemma~\ref{lemma:extremal}.

  Moreover, $q\acts \rho_q \in \Omega^+_\Hermitian$ is also strongly
  continuous, and so it follows from $\dupr{q\acts \rho_q}{b} = \dupr{\omega}{b}$
  for all $b\in \A$ that $q\acts \rho_q = \omega$. Consequently,
  \begin{equation*}
    \dupr{\omega-(a\acts\rho_q)}{b^*b} 
    =
    \dupr{\omega}{b^*b} - \dupr{\rho_q}{b^*a^*a\,b} 
    \ge
    \dupr{\omega}{b^*b} - \dupr{\rho_q}{b^*q^2 b}
    =
    \dupr{\omega}{b^*b} - \dupr{q\acts\rho_q}{b^*b}
    =
    0
  \end{equation*}
  holds for all $b\in \A$ and shows that $a\acts \rho_q \le \omega$ because
  $(\A,\Omega)$ is regular by the previous Corollary~\ref{corollary:regularsufcond}.
  But this yields 
  $a\acts \omega = \dupr{\rho_q}{\Unit}^{-1} (a \acts \rho_q) \le \dupr{\rho_q}{\Unit}^{-1} \omega$,
  where $\dupr{\rho_q}{\Unit}^{-1} > 0$ because $\rho_q \neq 0$ due to $q\acts \rho_q = \omega$,
  and then Proposition~\ref{proposition:extremal} shows that $\omega$ is multiplicative
  on $Q^\downarrow$.
  
  Finally, $\omega$ is multiplicative on whole $\A$ because 
  $Q^\downarrow$ is strongly dense in $\A$ and because 
  $\Var_\omega\colon \A \to \CC$ is strongly continuous and vanishes on 
  $Q^\downarrow$, hence on whole $\A$ (see Lemma~\ref{lemma:var}).
\end{proof}
\begin{example} \label{example:2}
  Let $\A$ be a locally convex $^*$\=/algebra and $(\A,\Omega)$ the corresponding
  abstract $O^*$\=/algebra like in Example~\ref{example:1}. If the product on $\A$
  is continuous, then $(\A,\Omega)$ is downwards closed. 
  Moreover, assume that there exists an
  upwards directed set of continuous seminorms $\mathcal{P}$ on $\A$, that defines the
  topology of $\A$, and a subset
  $Q'\subseteq \A^+_\Hermitian$ of pairwise commuting and coercive elements with the following properties:
  \begin{itemize}
    \item For all $\seminorm{p}{\argument} \in \mathcal{P}$ and all $q'\in Q'$,
    $\seminorm{p}{(q')^n} = 0$ holds for one $n\in \NN$ or $\sum_{n=1}^\infty \seminorm{p}{(q')^n}^{-1/(2n)} = \infty$.
    \item $Q^\downarrow$ is dense in $\A$, where $Q$ is the dominant set constructed
    out of $Q'$ like in Lemma~\ref{lemma:dominantconstruction}.
  \end{itemize}
  Then $(\A,\Omega)$ has sufficiently many Stieltjes states and 
  $\PureStates(\A,\Omega) = \Characters(\A,\Omega)$ by the previous 
  Theorem~\ref{theorem:unbounded} if $\A$ is commutative.
\end{example}
\begin{proof}
  Let a continuous algebraically positive linear functional $\omega$ on $\A$ be given
  as well as an algebraically positive linear functional $\rho$ on $\A$ such that
  $\omega-\rho$ is also algebraically positive. Then 
  $\abs{\dupr{\rho}{a}} \le \dupr{\rho}{\Unit}^{1/2} \dupr{\rho}{a^*a}^{1/2} \le \dupr{\rho}{\Unit}^{1/2} \dupr{\omega}{a^*a}^{1/2}$,
  and if the product on $\A$ is continuous, then the seminorm $\A \ni a \mapsto \seminorm{\omega,\st}{a} = \dupr{\omega}{a^*a}^{1/2}$
  is continuous, which shows that $\rho$ has to be continuous as well. So $(\A,\Omega)$
  is especially downwards closed and the topology on $\A$ is stronger than the strong one.
  
  If there exists an upwards directed set of continuous seminorms $\mathcal{P}$ on $\A$ and
  a subset $Q'\subseteq \A^+_\Hermitian$ with the properties stated above, then
  for every $\omega\in\States(\A,\Omega)$
  there are $\seminorm{p}{\argument}\in\mathcal{P}$ and $C\in[1,\infty[\,$ such that
  $\abs{\dupr{\omega}{a}} \le C\,\seminorm{p}{a}$ holds for all $a\in \A$. Especially
  for all $q'\in Q'$ this implies that $\dupr{\omega}{(q')^n} = 0$ holds for one $n\in \NN$
  (and thus for all $n\in \NN$ by Lemma~\ref{lemma:growth})
  or 
  \begin{equation*}
    \sum_{n=1}^\infty \dupr{\omega}{(q')^n}^{-1/(2n)} \ge \sum_{n=1}^\infty C^{-1/(2n)} \seminorm{p}{(q')^n}^{-1/(2n)}
    \ge C^{-1} \sum_{n=1}^\infty \seminorm{p}{(q')^n}^{-1/(2n)} = \infty\,,
  \end{equation*}
  so every state of $(\A,\Omega)$ is a Stieltjes state for every $q'\in Q'$.
  
  Finally, $(Q^\downarrow)^{++}_\Hermitian$ is dense in $\A^{++}_\Hermitian$, 
  hence in $\A^+_\Hermitian$ (see Example~\ref{example:1}):
  For every $\hat{a}\in \A$ there exists
  a net $(a_i)_{i\in I}$ in $Q^\downarrow$ over an upwards directed set $I$ that
  converges against $\hat{a}$, because $Q^\downarrow$ is dense in $\A$ by assumption,
  and so the net $(a_i^*a_i)_{i\in I}$ in 
  $(Q^\downarrow)^{++}_\Hermitian$ 
  converges against $\hat{a}^*\hat{a}$. As the strong topology on $\A$
  is weaker than the given one and 
  $(Q^\downarrow)^{++}_\Hermitian \subseteq (Q^\downarrow)^+_\Hermitian$, it follows
  that $(Q^\downarrow)^+_\Hermitian$ is strongly dense in $\A^+_\Hermitian$, i.e.
  $(\A,\Omega)$ has sufficiently many Stieltjes states.
\end{proof}
A locally convex, but not locally multiplicatively convex $^*$\=/algebra 
that fulfils the conditions of the above Example~\ref{example:2}
has been constructed in \cite{schoetz.waldmann:2018a}
(see the growth estimates for seminorms in Lemma~3.34 there)
in the context of non-formal deformation quantisation. 
This example describes the usual (convergent) star products of exponential type
for a space of finitely or infinitely many degrees of freedom and can contain 
elements $P,Q$ having canonical commutation relations $QP-PQ = \I \hbar \Unit$.
This rules out the existence of non-trivial submultiplicative seminorms in
the non-commutative case. Even in the commutative classical limit, the topology
does not become submultiplicative and hence the usual approach to pure states
like in Example~\ref{example:1} does not apply. Nevertheless, this case can still
be dealt with by the methods used in Example~\ref{example:2} and Theorem~\ref{theorem:unbounded}.

\bigskip

{
  \footnotesize
  \renewcommand{\arraystretch}{0.5}
  
}

\end{onehalfspace}

\begin{thebibliography}{10}
  
  \bibitem{beiser.waldmann:2014a}
  \textup{Beiser, S., Waldmann, S.: }
  \newblock
  \emph{Fr{\'{e}}chet algebraic deformation quantization of the Poincar{\'{e}} disk}
  \newblock
  Crelle's J. reine angew. Math. 688, 147-–207 (2014)
  
  \bibitem{reptheoforposfun}
  \textup{Bucy, R. S., Maltese, G.: }
  \newblock
  \emph{A Representation Theorem for Positive Functionals
  on Involution Algebras},
  \newblock
  Math. Annalen 162, 364--367 (1966)
  
  \bibitem{esposito.stapor.waldmann:2017a}
  \textup{Esposito, C., Stapor, P., Waldmann, S.: }
  \newblock
  \emph{Convergence of the Gutt Star Product}
  \newblock
  J. Lie Theory 27, 579–-622 (2017)
  
  \bibitem{gaur}
  \textup{Gaur, A. K.: }
  \newblock
  \emph{Pure states of the commutative Banach $^*$-algebras},
  \newblock
  Soochow J. of Math. 23, No. 4, 423--426 (1997)

  \bibitem{kaplansky}
  \textup{Kaplansky, I.: }
  \newblock
  \emph{Rings of Operators},
  \newblock
  W. A. Benjamin, Inc. New York, 1968

  \bibitem{koethe}
  \textup{K{\"{o}}the, G.: }
  \newblock
  \emph{Topologische Lineare Räume I},
  \newblock
  Springer-Verlag, Berlin, Heidelberg, New York, 1966

  \bibitem{masson.mcclary:1972}
  \textup{Masson, D., McClary, W. K.: }
  \newblock
  \emph{Classes of C$^\infty$ Vectors and Essential Self-Adjointness},
  \newblock
  J. Funct. Anal.  10, 19--32 (1972)
  
  \bibitem{neighbourhoodsofextremepoints}
  \textup{Namioka, I.: }
  \newblock
  \emph{Neighborhoods of extreme points},
  \newblock
  Israel J. Math. 5, 145--152 (1967)

  \bibitem{schmuedgen}
  \textup{Schm{\"{u}}dgen, K.: }
  \newblock 
  \emph{Unbounded Operator Algebras and Representation Theory},
  vol.~37 in \emph{Operator Theory: Advances and Applications}. 
  \newblock 
  Birkh{\"{a}}user Verlag, Basel, Boston, Berlin, 1990.
  
  \bibitem{schoetz.waldmann:2018a}
  \textup{Sch{\"o}tz, M., Waldmann, S.: }
  \newblock 
  \emph{Convergent star products for projective limits of Hilbert spaces},
  \newblock 
  J. Funct. Anal. 274, 1381--1423 (2018)
  
  \end{thebibliography}
\end{document}